\documentclass[onefignum,onetabnum]{siamart171218}

\usepackage{stmaryrd}

\usepackage{graphicx} 
\usepackage{amsfonts}
\usepackage{amssymb}
\usepackage{appendix}
\usepackage{subcaption}
\usepackage{float}
\def\vavg#1{\{\!\!\{#1\}\!\!\}}

\newcommand{\supp}{{\rm supp}}
\renewcommand{\div}{{\rm div}}

\newcommand{\bE}{{\boldsymbol E}}

\newcommand{\bI}{{\boldsymbol I}}

\newcommand{\bQ}{{\boldsymbol Q}}

\newcommand{\bp}{{\boldsymbol p}}
\newcommand{\bq}{{\boldsymbol q}}
\newcommand{\bw}{{\boldsymbol w}}
\newcommand{\bPi}{{\boldsymbol \Pi}}

\newcommand{\bn}{{\boldsymbol n}}

\newcommand{\be}{{\boldsymbol e}}
\newcommand{\bPhi}{{\boldsymbol \Phi}}
\newcommand{\bPsi}{{\boldsymbol \Psi}}

\newcommand{\bsigma}{{\boldsymbol \sigma}}
\newcommand{\bSigma}{{\boldsymbol \Sigma}}
\newcommand{\btau}{{\boldsymbol \tau}}
\newcommand{\bbeta}{{\boldsymbol \beta}}

\newsiamremark{remark}{Remark}
\newsiamremark{assumption}{Assumption}

\newsiamremark{hypothesis}{Hypothesis}
\crefname{hypothesis}{Hypothesis}{Hypotheses}
\newsiamthm{claim}{Claim}

\title{Local discontinuous Galerkin method for the integral fractional Laplacian
\thanks{Submitted to the editors DATE.
\funding{This work was supported in part by the National Natural Science Foundation of China (Grants No. 12222101, No. 12571383, and No. 12288101) and the Beijing Natural Science Foundation (Grant No. 1232007). }}
}

\author{
Rubing Han\thanks{School of Mathematical Sciences, Peking University, Beijing 100871, China (\email{hanrubing@pku.edu.cn}, \email{snwu@math.pku.edu.cn}, \email{zhouhao23@pku.edu.cn}).}
\and Shuonan Wu\footnotemark[2] 
\and Hao Zhou\footnotemark[2]
}

\begin{document}

\maketitle

\begin{abstract}
We develop and analyze a local discontinuous Galerkin (LDG) method for solving integral fractional Laplacian problems on bounded Lipschitz domains. The method is based on a three-field mixed formulation involving the primal variable, its gradient, and the corresponding Riesz potential, yielding a flux-based structure well suited for LDG discretizations while retaining the intrinsic nonlocal interaction. A key ingredient of our analysis is a rigorous study of the weighted H\"older and Sobolev regularity of the Riesz potential, which enables accurate characterization of boundary singularities. Guided by these regularity results, we propose LDG schemes on quasi-uniform and graded meshes, with additional stabilization in the graded case to reconcile the discrepancy between the discrete spaces for the Riesz potential and flux fields. Optimal a priori error estimates are established, and numerical experiments corroborate the theoretical results.
\end{abstract}

\section{Introduction} \label{sc:intro}
	Let $\Omega \subset\mathbb{R}^n$ be a bounded Lipschitz domain. In this paper we consider the following integral fractional Laplace equation with the homogeneous Dirichlet boundary condition: 
\begin{equation}\label{eq:primal_fractional_problem}
\left\{
\begin{aligned}
(-\Delta)^s u(x) = f(x), &\quad x \in \Omega,\\
u(x) = 0,\quad~               &\quad x \in \Omega^c := \mathbb{R}^n \setminus \Omega,
\end{aligned}
\right.
\end{equation}
where the integral fractional Laplacian of order $s\in(0,1)$ is defined by 
\[
(-\Delta)^s u(x) := C_{n,s} \text{P.V.} \int_{\mathbb{R}^n} \frac{u(x) - u(y)}{|x-y|^{n+2s}}\mathrm{d}y, \quad C_{n,s} = \frac{2^{2s} s \Gamma(s+\frac{n}{2})}{\pi^{n/2} \Gamma(1-s)}.
\]

The integral fractional Laplacian is a nonlocal generalization of the classical Laplacian and has been widely used to model phenomena involving long-range interactions in various applications \cite{tankov2003financial, constantin2005euler, caffarelli2011nonlinear}. 
A distinctive feature of the integral fractional Laplacian is that, even on smooth domains with smooth data, the solution to \eqref{eq:primal_fractional_problem} exhibits a boundary singularity of the form $\mathrm{dist}(x,\partial\Omega)^s$ \cite{ros2014dirichlet}. This intrinsic loss of regularity presents substantial challenges for numerical discretizations, and the accompanying error analysis must avoid relying on higher regularity assumptions that the solution does not satisfy.

A number of numerical methods that respect the intrinsic regularity of
the solution to \eqref{eq:primal_fractional_problem} have been
developed in recent years. In the variational framework, a conforming
finite element discretization based on continuous piecewise linear
spaces was proposed in \cite{acosta2017fractional}, where the weighted
H\"older regularity from \cite{ros2014dirichlet} was used to establish both standard and weighted Sobolev regularity, leading to optimal convergence in the energy norm on quasi-uniform and graded meshes. Further developments of this conforming approach include local error estimates \cite{borthagaray2021local,faustmann2022local}, multilevel solvers \cite{faustmann2021stability,borthagaray2023robust}, and extensions to the fractional obstacle problem \cite{borthagaray2019weighted}. In the finite difference setting, monotone schemes applicable to unstructured meshes were introduced in \cite{han2022monotone,han2023monotone}, where the weighted H\"older regularity was again essential for deriving optimal pointwise error bounds; the related two-scale discretization has also been extended to more general integro-differential operators \cite{borthagaray2024monotone}. The sinc-Galerkin method was interpreted as a nonconforming Galerkin scheme in \cite{antil2023analysis}, allowing for an analysis based on weighted Sobolev regularity. An exponentially convergent $hp$ finite element method was developed in \cite{faustmann2023exponential} using anisotropic graded meshes informed by the regularity theory in \cite{faustmann2022weighted}.
Comprehensive surveys on numerical methods and applications for the fractional Laplacian can be found in \cite{bonito2018numerical,lischke2020fractional,borthagaray2023fractional}.

While the methods reviewed above provide effective discretization strategies for the integral fractional Laplacian, they are mostly based on conforming or globally coupled approximation spaces. Discontinuous Galerkin (DG) formulations offer a flexible alternative: the use of discontinuous approximation spaces affords greater freedom in designing numerical fluxes and supports elementwise polynomial adaptivity. 
Over the past decades, DG methods have been systematically developed for second-order elliptic equations, leading to a mature framework with well-established stability theory and a priori error estimates. Representative examples include the interior penalty discontinuous Galerkin (IPDG)  methods \cite{babuvska1973nonconforming, arnold1982interior} and the local discontinuous Galerkin (LDG) methods for elliptic problems \cite{cockburn1998local}. In particular, an a priori error analysis for the LDG method was carried out in \cite{castillo2000priori}, where appropriate numerical fluxes were introduced to ensure stability and accuracy. A comprehensive survey of classical DG formulations for second-order elliptic equations was provided in \cite{arnold2002unified}.

Despite the extensive development of DG methods for second-order elliptic problems, the study of DG discretizations for the fractional Laplacian remains relatively limited. In fact, DG spaces are conforming for $s\in(0,\frac12)$, while for $s\in(\frac12,1)$ they retain their genuinely discontinuous character, which requires careful design of numerical schemes. In \cite{castillo2020convergence}, an LDG method for the spatial fractional Laplacian was investigated, where the formulation and analysis were carried out for one-dimensional problems on uniform meshes. Originally, the LDG methodology was developed for time-dependent systems \cite{cockburn1998local}. Subsequently, systematic constructions of LDG methods for time-dependent fractional diffusion problems were proposed in \cite{deng2013local, qiu2015nodal, nie2021local}. It is worth noting that the error estimates in these works are typically derived under relatively strong smoothness assumptions on the exact solution.

The objective of this work is to develop an LDG discretization for problem \eqref{eq:primal_fractional_problem} when $s\in(\tfrac12,1)$ and to derive error estimates under the intrinsic regularity of the solution. To this end, the problem is first recast, in the spirit of the LDG methodology, into a three-field mixed formulation involving the variables 
$$
u, \quad \boldsymbol{\sigma} :=\nabla u, \quad \text{and} \quad \boldsymbol{p} :=I_{2-2s}\boldsymbol{\sigma},$$ 
where $I_{2-2s}$ denotes certain Riesz potential (see
Definition~\ref{df:Riesz}). The weighted Sobolev regularity of the
Riesz potential is studied in terms of the H\"older regularity of the
data on Lipschitz domains satisfying the exterior ball condition. The
main technical difficulty consists in deriving weighted H\"older
estimates for the Riesz potential from the corresponding weighted
H\"older regularity of $\nabla u$ established in
\cite{ros2014dirichlet}. These results, together with the weighted
Sobolev regularity of $u$ and $\boldsymbol{\sigma}$ proven in
\cite{acosta2017fractional}, yield a complete regularity framework for
the mixed formulation.

Following the approach of \cite{castillo2000priori}, suitable numerical fluxes are introduced to construct the corresponding LDG scheme for the integral fractional Laplacian. A key feature of the present setting is that the above regularity results indicate different polynomial approximation orders should be employed in the DG spaces associated with $\boldsymbol{\sigma}$ and $\boldsymbol{p}$, reflecting their distinct differentiability properties. To accommodate this disparity, a stabilization term is incorporated to control the interaction between the two discrete vector-valued spaces. This leads to an LDG formulation for the fractional Laplacian that is structurally distinct from that used for the classical Laplacian.

A priori error estimates are established for the proposed LDG method on shape-regular triangulations. Projection operators adapted to the discontinuous finite element spaces, including local $L^2$ projections and Scott–Zhang interpolantion, are employed to derive the associated error equations. By invoking a discrete energy identity, which is standard in the analysis of DG methods, together with a duality argument, error bounds are obtained for the solution variable $u$ in mesh-dependent energy norms as well as in the $L^2$ norm. The analysis makes essential use of the relation between the Riesz potential and negative-order Sobolev norms. The resulting error bounds are ultimately reduced to appropriate interpolation estimates, in which the regularity results established above are systematically exploited to derive the corresponding convergence rates on quasi-uniform and graded meshes.

The remainder of this paper is organized as follows. In Section \ref{sc:preliminary}, preliminary results on Sobolev spaces and the Riesz potential are recalled, and the three-field mixed formulation is introduced. Section \ref{sc:regularity} is devoted to the regularity analysis of this formulation, with particular emphasis on the weighted H\"older and Sobolev regularity of the Riesz potential. In Section \ref{sc:discretization}, the LDG method is introduced and the main theoretical results of the paper are stated. The proofs of the error estimates on shape-regular triangulations are presented in Section \ref{sc:triangle}. Finally, numerical results are reported in Section \ref{sc:fracDG-numerical}.

We shall use $X \lesssim Y$ (resp. $X \gtrsim Y$) to denote $X \leq CY$ (resp. $X \geq CY$), where $C$ is a constant independent of the mesh size $h$. Additionally, $X \eqsim Y$ will signify that both $X \lesssim Y$ and $X \gtrsim Y$ hold.

\section{Preliminaries}\label{sc:preliminary}
In this section, we provide a concise overview of several key properties of Sobolev spaces that are essential for our analysis. Furthermore, we reformulate the problem using a mixed formulation, incorporating the so-called \emph{Riesz potential}, and examine its connection to Sobolev spaces.

\subsection{Sobolev spaces}
For $t \in \mathbb{R}$, the Sobolev space $H^t(\mathbb{R}^n)$ can be defined via the Fourier transform (cf. \cite{mclean2000strongly}):  
\[
H^t(\mathbb{R}^n) := \{u \in \mathcal{S}' : \|u\|_{H^t(\mathbb{R}^n)} < \infty\},
\]  
with the norm  
$\|u\|^2_{H^t(\mathbb{R}^n)} := \int_{\mathbb{R}^n} (1+|\xi|^2)^t |\widehat{u}(\xi)|^2 \mathrm{d}\xi$,  
where $\widehat{u} := \mathcal{F}[u]$ denotes the Fourier transform of $u$, and $\mathcal{S}'$ is the dual space of the Schwartz space $\mathcal{S}$.
 Given a Lipschitz domain $\Omega \subset \mathbb{R}^n$, define the spaces  
 \begin{equation} \label{eq:Ht}
 \begin{aligned}
 	H^t(\Omega) &:= \left\{u : u = U|_{\Omega} \text{ for some } U \in H^t(\mathbb{R}^n) \right\}, \\
 	\widetilde{H}^t(\Omega) &:= \left\{u \in H^t(\mathbb{R}^n) : \supp u \subset \overline{\Omega} \right\},
 \end{aligned}
 \end{equation}
 and equip them with the following norms:
 \[
 \|u \|_{H^t(\Omega)} := \inf\limits_{U|_{\Omega} = u} \|U\|_{H^t(\mathbb{R}^n)} \quad \text{and} \quad \|u\|_{\widetilde{H}^t(\Omega)} := \|u\|_{H^t(\mathbb{R}^n)}.
 \]


For $s > 0$, an equivalent norm for $H^s(\Omega)$ is given by  
\[
\|u\|_{H^s(\Omega)}^2 \eqsim \|u\|_{L^2(\Omega)}^2 + |u|_{H^s(\Omega)}^2,
\]  
where the seminorm $|\cdot|_{H^s(\Omega)}$ is defined as  
\[
|u|_{H^s(\Omega)}^2  :=
\left\{
\begin{aligned}
 \sum\limits_{|\alpha| = s} \|\partial^\alpha u \|_{L^2(\Omega)}^2, & \quad \text{if }s \in \mathbb{N}_+, \\
 \sum\limits_{|\alpha| = k} \iint_{\Omega\times \Omega} \frac{|\partial^\alpha u(x) - \partial^\alpha u(y)|^2}{|x-y|^{n+2\mu}}\mathrm{d}x\mathrm{d}y, & \quad \text{if }s = k+\mu, k \in \mathbb{N}, 0 < \mu < 1.
\end{aligned}
 \right.
\]

Throughout this work, boldface symbols, such as $\bp$ and $\boldsymbol{H}^t(\Omega)$, denote vector-valued functions and function spaces, while non-bold symbols, such as $u$ and $H^t(\Omega)$, represent their scalar-valued counterparts.

\subsection{Riesz potential}
For $0< s < 1$, the integral fractional Laplacian can be represented via its Fourier symbol as $\mathcal{F}[(-\Delta)^s u](\xi) = |\xi|^{2s} \mathcal{F}[u](\xi)$.
This expression can also be rewritten as:  
\[
\mathcal{F}[(-\Delta)^s u](\xi) = -i\xi \cdot |\xi|^{2s-2}(i\xi) \mathcal{F}[u](\xi) \quad 0 < s < 1.
\]
From the Fourier transform, it follows that the differential operators corresponding to $-i\xi \cdot$ and $i\xi$ are $-\div$ and $\nabla$, respectively. Furthermore, the operator whose Fourier symbol is $|\xi|^{2s-2}$ is commonly known as the \emph{Riesz potential} of order $2-2s$ \cite{landkof1972foundations}.
\begin{definition}[Riesz potential] \label{df:Riesz}
	For the order $0 < \alpha < n$, the Riesz potential of a function $v \in \mathcal{S}$ is defined as:  
	\begin{equation} \label{eq:Riesz}
	(I_\alpha v)(x) := \frac{1}{\gamma(\alpha)} \int_{\mathbb{R}^n} \frac{v(z)}{|x - z|^{n-\alpha}} \mathrm{d}z,
	\end{equation}
	where $\gamma(\alpha) =\pi^{\frac{n}{2}} 2^\alpha \Gamma(\frac{\alpha}{2})/ \Gamma\left(\frac{n}{2} - \frac{\alpha}{2}\right)$. 
\end{definition}

The Riesz potential satisfies the following properties (cf. \cite{stein1970singular}): For $f, g \in \mathcal{S}$,
\begin{enumerate}
\item For $0 < \alpha < n$, $\mathcal{F}(I_\alpha f) = |\xi|^{-\alpha} \mathcal{F}(f)$.
\item For $0 < \alpha, \beta < n$ and $0 < \alpha + \beta < n$, 
\begin{equation}\label{eq:Riesz-property}
	(I_{\alpha}I_{\beta} f,g)_{L^2(\mathbb{R}^n)} = (I_{\alpha+\beta}f,g)_{L^2(\mathbb{R}^n)}, ~
	(I_{\alpha} f,g)_{L^2(\mathbb{R}^n)} = (f,I_{\alpha}g)_{L^2(\mathbb{R}^n)}.
\end{equation}
\end{enumerate}

For $f \in L^2(\Omega)$ on a bounded domain $\Omega$, we defines its zero extension $\tilde{f}$ to $\mathbb{R}^n$. Then, the Riesz potential 
\eqref{eq:Riesz} is then expressed as:
\[
I_\alpha \tilde{f}(x) = \frac{1}{\gamma(\alpha)}\int_{\mathbb{R}^n} \frac{\tilde{f}(x-z)}{|z|^{n-\alpha}} \mathrm{d}z = \frac{1}{\gamma(\alpha)}\int_{|z| < R} \frac{\tilde{f}(x-z)}{|z|^{n-\alpha}} \mathrm{d}z = \frac{1}{\gamma(\alpha)}\left( \tilde{f} * \frac{\chi_{B_R}(z)}{|z|^{n-\alpha}} \right)(x),
\]
where $R > 0$ is chosen such that $\Omega \subset B_{\frac{R}{2}}$, and $\chi_{B_R}(z)$ is the characteristic function of the ball $B_R$.
Recall the definition of the Hardy-Littlewood maximal function:
\[
M\tilde{f}(x) := \sup_{r > 0} \frac{1}{|B_r(x)|} \int_{B_r(x)} |\tilde{f}(z)| \, \mathrm{d}z.
\]
Since $\frac{\chi_{B_R}(z)}{|z|^{n-\alpha}}$ is a radially decreasing function, it follows from \cite[Proposition 2.7, Chapter 2]{duoandikoetxea2024fourier} that   
\[
\begin{aligned}
|I_\alpha \tilde{f}(x)| &= \frac{1}{\gamma(\alpha)} \left| \left(\tilde{f} * \frac{\chi_{B_R}(z)}{|z|^{n-\alpha}}\right)(x) \right| \\
& \leq M\tilde{f}(x) \frac{1}{\gamma(\alpha)}\int_{\mathbb{R}^n} \frac{\chi_{B_R}(z)}{|z|^{n-\alpha}} \mathrm{d}z = C(n,\alpha) R^\alpha M\tilde{f}(x).
\end{aligned}
\]
Hence,  using the $L^2$-boundedness of the maximal operator, we obtain:
\begin{equation} \label{eq:Riesz-L2boundedness}
	\|I_\alpha \tilde{f}\|_{L^2(\mathbb{R}^n)} \leq C(n,\alpha) R^\alpha \|M\tilde{f}\|_{L^2(\mathbb{R}^n)} \leq C(n,\alpha) R^\alpha \|f\|_{L^2(\Omega)}.
\end{equation}

For convenience, in the following, we denote the Riesz potential of the zero extension of a function $f \in L^2(\Omega)$ simply as $I_\alpha f$, without distinguishing between the zero extension and the original function. Thanks to \eqref{eq:Riesz-L2boundedness} and the density argument, the above properties still hold for $f, g \in L^2(\Omega)$.

\subsection{A mixed formulation}
Let $\mathcal{F}[{\bsigma}] = i\xi \mathcal{F}[{u}]$ and $\mathcal{F}[{\bp}] = |\xi|^{2s-2} \mathcal{F}[{\bsigma}]$. Using these relations, the problem (\ref{eq:primal_fractional_problem}) can be reformulated into a mixed formulation: 
\begin{equation}\label{eq:mixed_fractional_problem}
	\left\{
	\begin{aligned}
		\bsigma - \nabla u &= 0,\quad \text{in } \Omega,\\
		\bp  -  I_{2-2s} \bsigma  &= 0,\quad \text{in } \mathbb{R}^n,\\
		-\div \bp &= f,\quad \text{in } \Omega,\\
		u & = 0,\quad  \text{in } \Omega^c.
	\end{aligned}
	\right.
\end{equation}
Since $u$ vanishes on $\Omega^c$, it follows that $\bsigma$ also vanishes there. Consequently, according to \eqref{eq:Riesz-L2boundedness}, $I_{2-2s} \bsigma|_{\Omega} \in L^2(\Omega)$ when $u$ satisfies $H^1$-regularity. Although $\bp$ is defined over the entire domain $\mathbb{R}^n$, the third equation in \eqref{eq:mixed_fractional_problem} only holds within $\Omega$. This feature allows the mixed formulation to effectively restrict all functions to $\Omega$ in the discrete setting, which is one of its key characteristics and advantages (see Section \ref{sc:discretization}).

We conclude this section by establishing the equivalence between the bilinear form induced by the Riesz potential and the standard Sobolev norm. 

\begin{lemma}[equivalent form of $\widetilde{H}^s$-norm]\label{lem:riesz_grad}  
    Given a bounded Lipschitz domain $\Omega$, let $\tilde{v}$ denote the zero extension of $v$ outside $\Omega$.  For $\frac12 < s< 1$, it holds that   
    \[
    \|I_{1-s} \nabla \tilde{v}\|_{L^2(\mathbb{R}^n)} \eqsim \|\tilde{v}\|_{\widetilde{H}^s(\Omega)}, \quad \forall v\in H_0^1(\Omega).
    \]
\end{lemma}

\begin{proof}
    By density argument, we only need to consider $v\in C_0^\infty(\Omega)$. Using \eqref{eq:Riesz-property} and Fourier transform of Riesz potential
    \[
    \|I_{1-s} \nabla \tilde{v}\|^2_{L^2(\mathbb{R}^n)} = (I_{2-2s}\nabla \tilde{v},\nabla \tilde{v}) = \frac{1}{(2\pi)^n} \int_{\mathbb{R}^n}|\xi|^{2s}|\mathcal{F}[\tilde{v}](\xi)|^2\mathrm{d}\xi \eqsim |\tilde{v}|_{H^s(\mathbb{R}^n)}^2.
    \]  
    Then, by the Poincaré inequality \cite[Proposition 2.4]{acosta2017fractional},  
    \[
    \|\tilde{v}\|_{L^2(\mathbb{R}^n)} = \|v\|_{L^2(\Omega)}\leq C(\Omega, n, s)|v|_{H^s(\mathbb{R}^n)}\quad \forall v\in \widetilde{H}^s(\Omega),
    \]
    we have  
    \[
    \|\tilde{v}\|_{\widetilde{H}^s(\Omega)}\lesssim \|I_{1-s} \nabla \tilde{v}\|_{L^2(\mathbb{R}^n)} \quad \forall v\in H_0^1(\Omega).
    \]
    The converse inequality is trivial by definition, which completes the proof.  
\end{proof}

\section{Sobolev regularity} \label{sc:regularity}
The regularity of the exact solution plays a fundamental role in the analysis of numerical errors. In this section, the Sobolev regularity theory for the integral fractional Laplacian is briefly reviewed, following the summary in \cite{acosta2017fractional, bonito2018numerical, borthagaray2021local}. Afterwards, we discuss the Sobolev regularity of $\bp := I_{2-2s}\nabla u$, which constitutes one of the main contributions of this work.

\subsection{Sobolev regularity of $u$ and $\bsigma$}
Since $\bsigma = \nabla u$, it suffices to present the regularity result for $u$. The following result is given in \cite[Eq.~(2.6)]{borthagaray2021local}.
\begin{theorem}[Sobolev regularity of $u$]\label{thm:sobolev_regularity}
Let $\Omega$ be a bounded Lipschitz domain, $\frac{1}{2} < s < 1$, and $f \in L^2(\Omega)$. Then the solution $u$ to problem~\eqref{eq:primal_fractional_problem} satisfies
	\begin{equation}\label{eq:sobolev-regularity-u}
		\|u\|_{H^{s+\frac12-\varepsilon}(\Omega)} \leq \frac{C(\Omega, n, s)}{\sqrt{\varepsilon}} \|f\|_{L^2(\Omega)}\quad \forall \;0< \varepsilon < s.
	\end{equation}
\end{theorem}

In \cite{acosta2017fractional}, weighted Sobolev spaces were introduced to describe the higher regularity of fractional Laplacian solutions near the boundary, where the weight is given by a power of the distance to $\partial \Omega$.  
Let  
\[
\delta(x):=\text{dist}(x,\partial\Omega), \quad \delta(x,y):=\min\{\delta(x),\delta(y)\}.
\]
For $k\in\mathbb{N}\cup\{0\}$ and $\gamma\ge0$, define the weighted Sobolev norm
\begin{equation*}
	\|v\|_{H^k_\gamma(\Omega)}^2 := \int_\Omega 
	\Big(|v(x)|^2 + \sum_{|\beta|\le k}|\partial^\beta v(x)|^2\Big)\delta(x)^{2\gamma}\,\mathrm{d}x.
\end{equation*}
The spaces $H^k_\gamma(\Omega)$ and $\widetilde H^k_\gamma(\Omega)$ are defined as the closures of $C^\infty(\Omega)$ and $C_0^\infty(\Omega)$, respectively, with respect to this norm. For $t=k+t'$ with $k\in\mathbb{N}\cup\{0\}$, $t'\in(0,1)$, and $\gamma\ge0$, set
\begin{equation}\label{eq:weighted_norm}
\begin{aligned}
	\|v\|_{H^t_\gamma(\Omega)}^2 &:= \|v\|_{H^k_\gamma(\Omega)}^2 + |v|_{H^t_\gamma(\Omega)}^2,\\
	|v|_{H^t_\gamma(\Omega)}^2 &:= 
	\iint_{\Omega \times \Omega}
	\frac{|\partial^k v(x)-\partial^k v(y)|^2}{|x-y|^{n+2t'}}\,
	\delta(x,y)^{2\gamma}\,\mathrm{d}y\,\mathrm{d}x,
\end{aligned}
\end{equation}
and define $H^t_\gamma(\Omega):=\{v\in H^k_\gamma(\Omega):\|v\|_{H^t_\gamma(\Omega)}<\infty\}$.
Analogously to the unweighted case, the zero-extension weighted Sobolev space is given by
\begin{equation}\label{eq:weighted_zero_space}
	\widetilde H^t_\gamma(\Omega):=
	\{v\in H^t_\gamma(\mathbb{R}^d): v=0 \text{ a.e. in } \Omega^c\},
\end{equation}
with the norm $\|v\|_{\widetilde H^t_\gamma(\Omega)}^2
:=\|v\|_{\widetilde H^k_\gamma(\Omega)}^2
+|v|_{H^t_\gamma(\mathbb{R}^d)}^2$. The following weighted regularity estimate then holds in $\widetilde H^t_\gamma(\Omega)$, see \cite[Theorem~2.3]{borthagaray2021local}.

\begin{theorem}[weighted Sobolev regularity of $u$]\label{thm:weighted_sobolev_estimate}
	Let $\Omega$ be a bounded, Lipschitz domain satisfying the exterior ball condition, $s\in(0,1)$, $f\in C^\beta(\overline{\Omega})$ for some $\beta  \in(0,2-2s),\; \gamma \geq0$, $t< \min\{ \beta +2s, \gamma + s +\frac{1}{2}\}$, and $u $ be the solution of (\ref{eq:primal_fractional_problem}). Then, it holds that $u\in \widetilde{H}_{\gamma}^t(\Omega)$ and 
	\begin{equation} \label{eq:weighted-regularity-u}
	\|u\|_{\widetilde{H}_\gamma^t(\Omega)} \leq \frac{C(\Omega, n, s)}{\sqrt{(\beta+2s-t)(\frac12+s + \gamma-t)}}\|f\|_{C^\beta(\overline{\Omega})} .
	\end{equation}
\end{theorem}

\begin{remark}[differentiability orders of $u$ and $\bsigma$] \label{rmk:differentiability-u-sigma}
Roughly speaking, when the right-hand side $f$ is sufficiently smooth (that is, when $\beta$ approaches $2-2s$ in Theorem~\ref{thm:weighted_sobolev_estimate}), the Sobolev differentiability order of $u$ increases with the weight parameter~$\gamma$, reaching up to approximately $s+\frac{1}{2}+\gamma$. Nevertheless, the overall regularity of $u$ does not exceed order~$2$, and the differentiability of $\bsigma=\nabla u$ remains below~$1$, regardless of the choice of the weight.
\end{remark}


Now, we discuss the unweighted Sobolev regularity of $\bp := I_{2-2s}\nabla u$, which follows from standard potential estimates.  
Let $w := \partial_i u \in H^{s-\frac{1}{2}-\varepsilon}(\Omega)$ for any  $0 < \varepsilon < s-\frac12$. Then,
\[
\|I_{2-2s}w\|_{H^{\frac{3}{2}-s-\varepsilon}(\Omega)} 
\leq \|I_{2-2s}\tilde{w}\|_{H^{\frac{3}{2}-s-\varepsilon}(\mathbb{R}^n)} 
\lesssim \|\tilde{w}\|_{H^{s-\frac{1}{2}-\varepsilon}(\mathbb{R}^n)} 
\lesssim \|w\|_{H^{s-\frac{1}{2}-\varepsilon}(\Omega)}.
\]
The second inequality follows from standard Fourier analysis on $\mathbb{R}^n$ \cite{stein1970singular}, while the last one is a consequence of the Sobolev extension theorem (noting that $0 < s-\frac{1}{2} < \frac{1}{2}$). By combining the above estimate with \eqref{eq:sobolev-regularity-u}, we obtain that for any $0 < \varepsilon < s-\tfrac{1}{2}$,
\begin{equation} \label{eq:sobolev-regularity-all}
\|\bsigma\|_{H^{s-\frac{1}{2}-\varepsilon}(\Omega)} 
+ \|\bp\|_{H^{\frac{3}{2}-s-\varepsilon}(\Omega)} 
+ \|u\|_{H^{s+\frac{1}{2}-\varepsilon}(\Omega)} 
\leq \frac{C(\Omega, n, s)}{\sqrt{\varepsilon}} \|f\|_{L^2(\Omega)}.
\end{equation}

\subsection{Weighted H\"{o}lder regularity of $\bp$}
The study of the weighted regularity for $\bp := I_{2-2s}\nabla u$ follows a similar spirit to the weighted fractional regularity theory developed in \cite{acosta2017fractional}, on which Theorem~\ref{thm:weighted_sobolev_estimate} is also based. However, the analysis here involves additional technical challenges due to the interplay between the nonlocal operator $I_{2-2s}$ and the boundary singularity of $\nabla u$. As in the classical approach, the starting point is the weighted Hölder estimate for $u$, originally established in \cite{ros2014dirichlet}. Specifically, from \cite[Proposition 1.4]{ros2014dirichlet}, one has for any $\beta\in(0,2-2s)$
\begin{equation*}
\sup_{x\in\Omega}\delta(x)^{1-s}|\nabla u(x)| + 
\sup_{x,y\in\Omega}
\delta(x,y)^{\beta+s} \frac{|\nabla u(x)-\nabla u(y)|}{|x-y|^{\beta+2s-1}}
\leq C(\Omega,s,\beta) \|f\|_{C^{\beta}(\overline{\Omega})}.
\end{equation*}
This corollary is also explicitly stated in \cite[Remark~3.4]{acosta2017fractional}.  
In that reference, the right-hand side is measured in the weighted Hölder norm $\|f\|_{\beta}^{(s)}$.  
Here, for simplicity, we employ the unweighted norm $\|f\|_{C^\beta(\overline{\Omega})}$ instead, which yields a slightly stronger condition and avoids introducing additional notation.

We now set $w = \partial_i u$.  
Then, the above weighted H\"{o}lder estimate above can be expressed in the form: for any $\beta \in (0,2-2s)$
\begin{subequations}\label{eq:weighted_holder_w_all}
\begin{align}
|w(x)| &\leq C(\Omega,s,\beta)\, \|f\|_{C^\beta(\overline{\Omega})}\, \delta(x)^{s-1}
&& \forall\, x\in\Omega, \label{eq:weighted_holder_w_a}\\[4pt]
|w(x) - w(y)| &\leq C(\Omega,s,\beta)\, \|f\|_{C^\beta(\overline{\Omega})}\,
\delta(x,y)^{-\beta - s}\, |x - y|^{\beta + 2s - 1}
&& \forall\, x,y\in\Omega.
\label{eq:weighted_holder_w_b}
\end{align}
\end{subequations}

\begin{lemma}[differentiability of the Riesz potential]\label{lem:riesz_diff}
Let $w$ satisfy the weighted Hölder estimates in \eqref{eq:weighted_holder_w_all}.  
Then the gradient of the Riesz potential $I_{2-2s}w$ exists pointwise for all $x\in\Omega$ and is given by
\begin{equation}\label{eq:riesz_grad}
\nabla I_{2-2s}w(x)
 = \tilde{C}_{n,s} \mathrm{P.V.} \int_{\Omega} 
   \frac{w(z)(x-z)}{|x-z|^{n+2s}}\mathrm{d}z, ~\tilde{C}_{n,s} := \frac{(2-2s-n)\Gamma(\frac{n}{2} -  1 + s)}{\pi^{\frac{n}{2}} 2^{2-2s}\Gamma(1-s)}.
\end{equation} 
Moreover, $\nabla I_{2-2s}w$ satisfies the pointwise estimate
\begin{equation}\label{eq:riesz_grad_est}
|\nabla I_{2-2s}w(x)| \le C(\Omega,s,\beta)\|f\|_{C^\beta(\overline{\Omega})}\delta(x)^{-s}
\quad \forall\, x\in\Omega.
\end{equation}
\end{lemma}
\begin{proof}
We denote $\alpha := 2 - 2s$ and the kernel $K(x,z) := \frac{1}{\gamma(\alpha)|x-z|^{n-\alpha}}$. For any fixed $x \in \Omega$, let $\delta := \delta(x)$.
It is straightforward to verify that the kernel satisfies the following Lipschitz-type estimate:
\begin{equation}\label{eq:kernel-Lipschitz}
|K(x,z) - K(y,z)| \leq C(n,\alpha)
\frac{|x-y|}{\min\{|x-z|, |y-z|\}^{n-\alpha+1}} \quad \forall x, y, z \in \Omega. 
\end{equation}
Since the singularity of $I_{\alpha}w(x)$ arises both at point $x$ and near the boundary $\partial\Omega$,  
the integration domain is decomposed into three regions (see, e.g., the similar decomposition in \cite{han2022monotone}):
\begin{equation} \label{eq:domain-decomposition}
R_1 := \{z \in \Omega : \delta(z) \leq \tfrac{\delta}{2}\}, \quad R_2 := B_{\frac{\delta}{2}}(x), \quad R_3 := \Omega \setminus (R_1 \cup R_2).
\end{equation}
Note that for any $z \in R_2$, we have $\delta(z) \eqsim \delta \eqsim \delta(x,z)$, since
\begin{equation*}
\delta(z) = \mathrm{dist}(z, \partial \Omega) \ge \tfrac{1}{2}\delta, \quad \text{and}
\quad
\delta = |x^* - x| \ge |x^* - z| - |x - z| \ge \delta(z) - \tfrac{1}{2}\delta,
\end{equation*}
where $x^*$ denotes the boundary point realizing $\delta = \mathrm{dist}(x, \partial \Omega)$.

Next, for any $y \in B_{\frac{\delta}{4}}(x)$, we estimate the difference 
$I_\alpha w(x) - I_\alpha w(y)$ by decomposing the integration domain into 
the three regions defined above. 

\fbox{Case 1: $z \in R_1$.} The singularity originates from the factor
  $\delta(z)$, and moreover $|x - z| \eqsim |y - z|$. Hence,
\begin{align*}
& \quad \left|
\int_{R_1} \big(K(x,z) - K(y,z) \big) w(z) \mathrm{d}z\right| \\
&\le C(\alpha, n) |x - y| 
\int_{R_1} \frac{|w(z)|}{|x - z|^{n - \alpha + 1}} \mathrm{d}z \\
&\le C(\Omega, s, \beta) \|f\|_{C^\beta(\overline{\Omega})} |x-y|
\int_{R_1} \frac{\delta(z)^{s - 1}}{|x - z|^{n - \alpha + 1}} \, \mathrm{d}z.
\end{align*}
To estimate the last integral term, we first consider the case in which
$\Omega = \mathbb{R}^n_+ := \{(z', z_n): z' \in \mathbb{R}^{n-1},\, z_n > 0\}$.  
Without loss of generality, let $x = (0,\ldots,0,\delta)$ and 
$R_1 = \mathbb{R}^{n-1} \times (0, C_1 \delta)$ for some $0 < C_1 < 1$.  
Then, we have
\begin{equation*}
\begin{aligned}
& \quad \int_{R_1} \frac{\delta(z)^{s-1}}{|x - z|^{n-\alpha+1}}\mathrm{d}z \\
&= \int_0^{C_1\delta} \int_{\mathbb{R}^{n-1}} z_n^{s-1} |x - z|^{\alpha - n - 1} \mathrm{d}z'\mathrm{d}z_n \\
&= |\mathbb{S}^{n-1}| \int_0^{C_1\delta} z_n^{s-1} \int_0^{\infty} \big((\delta-z_n)^2 + r^2\big)^{\frac{\alpha - n - 1}{2}} r^{n-1}\mathrm{d}r \mathrm{d}z_n \\
&\leq C(n) \int_0^{C_1\delta} z_n^{s-1} 
\Big( \int_0^{\delta - z_n} (\delta - z_n)^{\alpha - n - 1} r^{n-1}\,\mathrm{d}r 
+ \int_{\delta - z_n}^{\infty} r^{\alpha - 2}\,\mathrm{d}r \Big) \mathrm{d}z_n \\
&\leq C(n,\alpha) \int_0^{C_1\delta} z_n^{s-1} (\delta - z_n)^{\alpha - 1} \mathrm{d}z_n.
\end{aligned}
\end{equation*}
Here, $\mathbb{S}^{n-1}$ denotes the unit sphere in $\mathbb{R}^n$.
Since $\alpha = 2 - 2s < 1$, the integral $\int_{\delta - z_n}^{\infty} r^{\alpha - 2}\mathrm{d}r$ is convergent. Performing the transformation $t = z_n / \delta$ yields
\[
\int_0^{C_1\delta} z_n^{s-1} (\delta - z_n)^{\alpha -1} \mathrm{d}z_n
= \delta^{s+\alpha -1}\int_0^{C_1} t^{s-1}(1-t)^{\alpha -1} \mathrm{d}t
\lesssim \delta^{1-s}.
\]
For a general bounded Lipschitz domain $\Omega$, the same estimate follows by a standard partition-of-unity argument combined with a bi-Lipschitz coordinate transformation.  
Consequently, we obtain
\begin{equation}\label{eq:R1}
\left|
\int_{R_1} \big(K(x,z) - K(y,z)\big) w(z)\mathrm{d}z
\right|
\le C(\Omega, s, \beta) \|f\|_{C^\beta(\overline{\Omega})} \delta^{1-s} |x-y|.
\end{equation}

\fbox{Case 2: $z \in R_2$.} We decompose the integral into two parts:
\[
\begin{aligned}
& \int_{R_2} \big(K(x,z) - K(y,z)\big) w(z)\mathrm{d}z \\
=~& w(x) \underbrace{\int_{R_2}  K(x,z) - K(y,z) \mathrm{d}z}_{=: I_{2,1}}
+ \underbrace{\int_{R_2} \big(K(x,z) - K(y,z)\big) \big(w(z) - w(x)\big)\mathrm{d}z}_{=: I_{2,2}}.
\end{aligned}
\]
For the first term $I_{2,1}$, the symmetry of the kernel implies
\[
\begin{aligned}
I_{2,1}
&= \frac{1}{\gamma(\alpha)}
\left(
\int_{B_{\frac{\delta}{2}}(x)} \frac{1}{|x-z|^{n-\alpha}}\mathrm{d}z
- \int_{B_{\frac{\delta}{2}}(x)} \frac{1}{|y-z|^{n-\alpha}}\mathrm{d}z
\right) \\
&= \frac{1}{\gamma(\alpha)}
\left(
\int_{B_{\frac{\delta}{2}}(x)} 
- \int_{B_{\frac{\delta}{2}}(y)}
\right)
\frac{1}{|x-z|^{n-\alpha}}\mathrm{d}z \\
&= \frac{1}{\gamma(\alpha)} 
\int_{B_{\frac{\delta}{2}}(x) \setminus B_{\frac{\delta}{2}}(y)}
\frac{1}{|x-z|^{n-\alpha}} - \frac{1}{|y-z|^{n-\alpha}}\mathrm{d}z.
\end{aligned}
\]
\noindent
Since $|x - y| < \frac{\delta}{4}$, for any 
$z \in B_{\frac{\delta}{2}}(x) \setminus B_{\frac{\delta}{2}}(y)$, it holds that $\min\{|z - x|,\, |z - y|\} \ge \frac{\delta}{4}$.
Moreover, the measure of this region satisfies $|B_{\frac{\delta}{2}}(x) \setminus B_{\frac{\delta}{2}}(y)| 
\le C(n)\, \delta^{n-1} |x - y|$. Consequently, in light of \eqref{eq:kernel-Lipschitz} and \eqref{eq:weighted_holder_w_a}, we have $|I_{2,1}| \le C(n, s) \delta^{\alpha-2} |x - y|^2$ and hence
\begin{equation} \label{eq:R2-I1}
\left|w(x) \int_{R_2}  K(x,z) - K(y,z) \mathrm{d}z \right| \le C(\Omega, s, \beta) \delta^{-s-1} \|f\|_{C^\beta(\overline{\Omega})} |x-y|^2.
\end{equation}

For the second term $I_{2,2}$, we apply the weighted H\"{o}lder estimate \eqref{eq:weighted_holder_w_b} to obtain
$$
\begin{aligned}
|I_{2,2}| & \leq C(\Omega, s, \beta) \|f\|_{C^\beta(\overline{\Omega})} \delta^{-\beta -s} \int_{B_{\frac{\delta}{2}}(x)} |x-z|^{\beta+2s-1}|K(x,z) - K(y,z)| \mathrm{d}z.
\end{aligned}
$$
Based on the distinct singular behaviors of the kernel in different regions, we divide the ball $B_{\frac{\delta}{2}}(x)$ into two parts:
$B_{2|x-y|}(x)$ and $B_{\frac{\delta}{2}}(x) \setminus B_{2|x-y|}(x)$. First, we have
$$
\begin{aligned}
& \int_{B_{2|x-y|}(x)} |x-z|^{\beta + 2s - 1} |K(x,z) - K(y,z)| \mathrm{d}z \\
\lesssim & \int_{B_{2|x-y|}(x)} |x-z|^{\beta + 2s - 1} K(x,z) \mathrm{d}z  + |x-y|^{\beta + 2s - 1} \int_{B_{2|x-y|}(x)}  K(y,z) \mathrm{d}z \\
\lesssim & \int_{B_{2|x-y|}(x)} |x-z|^{\beta + 1 - n} \mathrm{d}z + |x-y|^{\beta + 2s - 1} \int_{B_{3|x-y|}(y)} |y-z|^{2-2s-n} \mathrm{d}z \\
\lesssim & \int_0^{2|x-y|} r^{\beta} \mathrm{d}r + |x-y|^{\beta + 2s - 1} \int_0^{3|x-y|} r^{1-2s} \mathrm{d}r \leq C(s) |x-y|^{\beta + 1}.
\end{aligned}
$$ 
Next, noting that for $z \in B_{\frac{\delta}{2}}(x) \setminus B_{2|x-y|}(x)$, we have $2|y-z| \ge |y-z| + |x-z| - |y-x| \ge |x-z|$. It then follows from \eqref{eq:kernel-Lipschitz} that
$$
\begin{aligned}
& \int_{B_{\frac{\delta}{2}}(x) \setminus B_{2|x-y|}(x)} |x-z|^{\beta + 2s - 1} |K(x,z) - K(y,z)| \mathrm{d}z \\
\lesssim & \int_{B_{\frac{\delta}{2}}(x) \setminus B_{2|x-y|}(x)}  |x-z|^{\beta+2s-1} \frac{|x-y|}{|x-z|^{n-\alpha+1}}  \mathrm{d}z  \\
\lesssim &\, |x-y|\int_{2|x-y|}^{\frac{\delta}{2}} r^{\beta -1} \mathrm{d}r \leq C(s,\beta) \delta^{\beta}|x-y|.
\end{aligned}
$$ 
It should be emphasized that the above bound is valid for any $\beta>0$.

Substituting the above two estimates into the bound for $I_{2,2}$ yields
\begin{equation} \label{eq:R2-I2}
\left|  \int_{R_2} \big(K(x,z) - K(y,z)\big) \big(w(z) - w(x)\big) \mathrm{d}z \right| \leq C(\Omega,s,\beta)\|f\|_{C^\beta(\overline{\Omega})} \delta^{-s} |x-y|.
\end{equation}

\fbox{Case 3: $z \in R_3$.} In this case, the integral is nonsingular, and it holds that
$\delta(z) \gtrsim \delta$ and $\min{|x-z|, |y-z|} \eqsim |x-z|$.
Therefore, by applying \eqref{eq:kernel-Lipschitz} and \eqref{eq:weighted_holder_w_a}, we obtain
\begin{equation} \label{eq:R3}
\begin{aligned}
& \quad \left|\int_{R_3} (K(x,z) - K(y,z))w(z) \mathrm{d}z\right| \\
&\leq C(\Omega,s,\beta)\|f\|_{C^\beta(\overline{\Omega})} \int_{R_3} \delta(z)^{s-1} \frac{|x-y|}{|x-z|^{n-\alpha+1}} \mathrm{d}z \\
& \leq C(\Omega,s,\beta)\|f\|_{C^\beta(\overline{\Omega})} \delta^{s-1}|x-y|\int_{\frac{\delta}{2}}^\infty  r^{-2s} \mathrm{d} r \\
& \leq C(\Omega,s,\beta)\|f\|_{C^\beta(\overline{\Omega})} \delta^{-s}|x-y|. \\
\end{aligned}
\end{equation}

Note that the kernel $K(x,z)$ is smooth for $z \neq x$, using the techniques from Case 2, it is straightforward to see that the truncated integral
\[
\begin{aligned}
& \quad \int_{\Omega \setminus B_\varepsilon(x)} \nabla_x K(x,z)w(z) \mathrm{d}z \\
&= w(x) \int_{\Omega \setminus B_\varepsilon(x)} \nabla_x K(x,z) \mathrm{d}z 
+ \int_{\Omega \setminus B_\varepsilon(x)} \nabla_x K(x,z)\, \big(w(z) - w(x)\big) \mathrm{d}z
\end{aligned}
\]  
admits a well-defined limit as $\varepsilon \to 0^+$, i.e., the Cauchy principal value integral on the right-hand side of \eqref{eq:riesz_grad}.
This is because the integral of the first term over the ball vanishes due to the symmetry of the kernel, while the second term can be estimated using the weighted H\"older estimate \eqref{eq:weighted_holder_w_b}: as $\varepsilon \to 0^+$,
\[
\Big|\int_{B_\varepsilon(x)} \nabla_x K(x,z)\, \big(w(z) - w(x)\big)\mathrm{d}z \Big|
\lesssim \delta^{-\beta - s} \int_0^{\varepsilon} r^{\beta - 1} \mathrm{d} r
\lesssim \delta^{-\beta - s} \varepsilon^\beta \to 0.
\]
Combining this with the estimates in \eqref{eq:R1}, \eqref{eq:R2-I1}, \eqref{eq:R2-I2}, and \eqref{eq:R3}, and applying the dominated convergence theorem, one immediately obtains the gradient representation \eqref{eq:riesz_grad} as well as the pointwise bound \eqref{eq:riesz_grad_est}.
\end{proof}

By using the analysis techniques developed in the above lemma, we can also obtain a weighted H\"{o}lder estimate for $\nabla I_{2-2s} w$.

\begin{lemma}[weighted H\"{o}lder estimate of $\nabla I_{2-2s}w$]  \label{lem:riesz-gradient-Holder}
Let $w$ satisfy the weighted Hölder estimates in \eqref{eq:weighted_holder_w_all}.  
For any $y \in B_{\frac{\delta(x)}{4}}(x)$, we have 
\begin{equation} \label{eq:riesz-gradient-Holder}
\left|\nabla I_{2-2s}w (x) - \nabla I_{2-2s}w (y) \right| \leq C(\Omega, s, \beta) \|f\|_{C^\beta(\overline{\Omega})} \delta(x)^{-s-\beta} |x-y|^{\beta}.
\end{equation}
\end{lemma}
\begin{proof}
We continue to use the notation introduced in Lemma \ref{lem:riesz_diff}, where $\alpha := 2 - 2s$ and $\delta := \delta(x)$. 
Denote
\[
\tilde{K}(x,z) := \nabla_x K(x,z) = \tilde{C}_{n,s} \frac{(x-z)}{|x-z|^{n-\alpha+2}}.
\]
It is straightforward to verify that
\begin{equation} \label{eq:tildeK-Lipschitz}
|\tilde{K}(x,z) - \tilde{K}(y,z)| \le C(n,s)\,
\frac{|x-y|}{\min\{|x-z|, |y-z|\}^{n-\alpha+2}}.
\end{equation}

The proof follows similar lines to that of Lemma \ref{lem:riesz_diff}, with several technical modifications.
We consider the difference $\nabla I_{2-2s}w(x) - \nabla I_{2-2s}w(y)$ over the different integration regions. 
The only minor modification is that, in the domain decomposition \eqref{eq:domain-decomposition}, 
the set $R_2$ is adjusted by removing $B_\varepsilon(x)$ (and $B_\varepsilon(y)$ for the part associated with $y$), 
for sufficiently small $\varepsilon>0$.

\fbox{Case 1}: For $z \in R_1$, applying \eqref{eq:tildeK-Lipschitz} yields
\begin{equation} \label{eq:grad-R1}
\begin{aligned}
& \left|
\int_{R_1} \big(\tilde{K}(x,z) - \tilde{K}(y,z) \big) w(z) \mathrm{d}z\right|  \\
\le&~  C(\Omega, s, \beta) \|f\|_{C^\beta(\overline{\Omega})} |x-y|
\int_{R_1} \frac{\delta(z)^{s - 1}}{|x - z|^{n - \alpha + 2}} \mathrm{d}z \\
\le&~ C(\Omega, s, \beta) \|f\|_{C^\beta(\overline{\Omega})} \delta^{-s} |x-y|.
\end{aligned}
\end{equation} 

\fbox{Case 2}: For any $\varepsilon < \frac{|x-y|}{2}$, it holds that 
$$
\begin{aligned}
& \quad \int_{B_{\frac{\delta}{2}}(x) \setminus B_\varepsilon(x)} \tilde{K}(x,z)w(z) \mathrm{d}z - 
\int_{B_{\frac{\delta}{2}}(x) \setminus B_\varepsilon(y)} \tilde{K}(y,z)w(z) \mathrm{d}z \\
= &~ \int_{B_{\frac{\delta}{2}}(x) \setminus B_\varepsilon(x)} \tilde{K}(x,z) \big(w(z) - w(x)\big) \mathrm{d}z 
- \int_{B_{\frac{\delta}{2}}(x) \setminus B_\varepsilon(y)} \tilde{K}(y,z) \big(w(z) - w(y)\big) \mathrm{d}z \\
& + w(y) \left( \int_{B_{\frac{\delta}{2}}(y) \setminus B_\varepsilon(y)} - \int_{B_{\frac{\delta}{2}}(x) \setminus B_\varepsilon(y)}\right) \tilde{K}(y,z) \mathrm{d}z \\
= &~ \int_{B_{2|x-y|}(x) \setminus B_\varepsilon(x)} \tilde{K}(x,z) \big(w(z) - w(x)\big) \mathrm{d}z \\
   &~ - \int_{B_{2|x-y|}(x) \setminus B_\varepsilon(y)} \tilde{K}(y,z) \big(w(z) - w(y)\big) \mathrm{d}z \\
   & + \int_{B_{\frac{\delta}{2}}(x) \setminus B_{2|x-y|}(x)} \tilde{K}(x,z) \big(w(z) - w(x)\big) - \tilde{K}(y,z) \big(w(z) - w(y)\big) \mathrm{d}z \\
   & + w(y) \left( \int_{B_{\frac{\delta}{2}}(y) \setminus B_\varepsilon(y)} - \int_{B_{\frac{\delta}{2}}(x) \setminus B_\varepsilon(y)}\right) \tilde{K}(y,z) \mathrm{d}z := J_1 + J_2 + J_3 +J_4.
\end{aligned}
$$ 
\begin{itemize}
\item Using \eqref{eq:weighted_holder_w_a}, we obtain
$$
\begin{aligned}
|J_1| & \leq C(\Omega,s,\beta)\|f\|_{C^\beta(\overline{\Omega})} \delta^{-\beta-s} \int_{B_{2|x-y|}(x) \setminus B_\varepsilon(x) }  |x-z|^{\beta - n} \mathrm{d}z\\
& \leq C(\Omega,s,\beta)\|f\|_{C^\beta(\overline{\Omega})} \delta^{-\beta-s} |x-y|^{\beta}.
\end{aligned}
$$ 
The same bound applies to $J_2$, as its estimate follows identically.

\item For $z \in B_{\frac{\delta}{2}}(x) \setminus B_{2|x-y|}(x)$, we have
$$
\begin{aligned}
& \left|\tilde{K}(x,z) \big(w(z) - w(x)\big) - \tilde{K}(y,z) \big(w(z) - w(y)\big)\right| \\
\leq &~ |\tilde{K}(x,z) - \tilde{K}(y,z)| \cdot |w(z) - w(x)| + |\tilde{K}(y,z)|\cdot |w(x) - w(y)| \\
\leq &~ C(\Omega,s,\beta)\|f\|_{C^\beta(\overline{\Omega})} \delta^{-\beta-s} \big( |x-y| |x-z|^{\beta -1 -n}  \\
&~\qquad \qquad \qquad \qquad\qquad+  |x-y|^{\beta + 2s -1}  |x-z|^{1-2s-n}\big).
\end{aligned}
$$ 
Therefore, 
$$
\begin{aligned}
|J_3| & \leq C(\Omega,s,\beta)\|f\|_{C^\beta(\overline{\Omega})}\delta^{-\beta-s} \Big(|x-y| \int_{2|x-y|}^{\frac{\delta}{2}} r^{\beta -2} \mathrm{d}r \\
& \qquad\qquad\qquad\qquad\qquad + |x-y|^{\beta + 2s -1} \int_{2|x-y|}^{\frac{\delta}{2}} r^{-2s}\mathrm{d}r \Big) \\
& \leq C(\Omega,s,\beta)\|f\|_{C^\beta(\overline{\Omega})}\delta^{-\beta-s} |x-y|^{\beta}.
\end{aligned}
$$ 
Here, we use the fact that $\beta -2 < -1$ and $-2s < -1$.

\item By again using the estimate 
$|B_{\frac{\delta}{2}}(y) \setminus B_{\frac{\delta}{2}}(x)| \le C(n)\delta^{n-1}|x-y|$ 
and applying \eqref{eq:weighted_holder_w_a}, we obtain $|J_4| \leq C(\Omega,s,\beta)\|f\|_{C^\beta(\overline{\Omega})} \delta^{-s-1}|x-y|$.
\end{itemize}

Combining the estimates for $J_1$–$J_4$, we arrive at the following bound,
which is uniform with respect to $\varepsilon$:
\begin{equation} \label{eq:grad-R2}
\begin{aligned}
& \left|\int_{B_{\frac{\delta}{2}}(x) \setminus B_\varepsilon(x)} \tilde{K}(x,z)w(z) \mathrm{d}z - 
\int_{B_{\frac{\delta}{2}}(x) \setminus B_\varepsilon(y)} \tilde{K}(y,z)w(z) \mathrm{d}z\right| \\
\leq &~ C(\Omega,s,\beta)\|f\|_{C^\beta(\overline{\Omega})}\delta^{-\beta-s} |x-y|^{\beta}.
\end{aligned}
\end{equation}

\fbox{Case 3}: For $z \in R_3$, the argument is similar to that of Lemma \ref{lem:riesz_diff}, which yields
\begin{equation} \label{eq:grad-R3}
 \left|\int_{R_3} \big(\tilde{K}(x,z) - \tilde{K}(y,z) \big)w(z) \mathrm{d}z\right| 
 \leq C(\Omega,s,\beta)\|f\|_{C^\beta(\overline{\Omega})} \delta^{-s-1}|x-y|.
\end{equation}

Combining the estimates \eqref{eq:grad-R1}, \eqref{eq:grad-R2}, and \eqref{eq:grad-R3}, 
and letting $\varepsilon \to 0^+$, we conclude the desired result \eqref{eq:riesz-gradient-Holder}.
\end{proof}

\subsection{Weighted Sobolev regularity of $\bp$}

Since $w = \partial_i u$, it follows that $p_i = I_{2-2s}w$.  
Based on the pointwise estimate for $\nabla I_{2-2s}w$ (Lemma \ref{lem:riesz_diff}) and its weighted Hölder regularity (Lemma \ref{lem:riesz-gradient-Holder}), 
we proceed by adopting the analytical framework developed in \cite{acosta2017fractional}.  
For this purpose, we define the subsets
\begin{equation*}
A := \{(x,y) \in \Omega \times \Omega : \delta(x,y) = \delta(x)\}, \quad 
B := \{(x,y) \in A : |x-y| \ge \tfrac{1}{4}\delta(x)\}.
\end{equation*}
We also recall a useful identity concerning the integrability of powers of the distance function $\delta(x)$; see, for instance, \cite[Remark 3.5]{acosta2017fractional}:
\begin{equation} \label{eq:delta-integral}
\int_{\Omega} \delta(x)^{-\alpha} \mathrm{d}x = \mathcal{O}(\frac{1}{1-\alpha}) \quad \forall \alpha < 1.
\end{equation}

\begin{theorem}[weighted Sobolev regularity of $\bp$]\label{thm:weighted_sobolev-p}
	Let $\Omega$ be a bounded, Lipschitz domain satisfying the exterior ball condition, $s\in(\frac12,1)$, $f\in C^\beta(\overline{\Omega})$ for some $\beta  \in(0,2-2s), \gamma \geq0$, $t< \min\{ \beta + 1, \gamma + \frac{3}{2}-s\}$, and $u $ be the solution of (\ref{eq:primal_fractional_problem}). Then, it holds that $\bp := I_{2-2s}\nabla u \in \boldsymbol{H}_{\gamma}^t(\Omega)$ and 
\begin{equation} \label{eq:weighted-regularity-p}
\|\bp\|_{H_\gamma^t(\Omega)} \leq 
\dfrac{C(\Omega, s, \beta)}{\sqrt{(\beta+1-t)\left(\tfrac{3}{2}-s+\gamma-t\right)|t-1|}}\|f\|_{C^\beta(\overline{\Omega})}.
\end{equation}
\end{theorem}
\begin{proof}
\fbox{Case 1: $\gamma \leq s - \frac12$}. We examine the condition under which $|\boldsymbol{p}|_{H^t_\gamma} < \infty$, 
namely, the integrability of the following quantity:
$$ 
\begin{aligned}
I &= \iint_{\Omega \times \Omega} \frac{|\bp(x) - \bp(y)|^2}{|x-y|^{n+2t}} \delta(x,y)^{2\gamma} \mathrm{d}x\mathrm{d}y \\
& = 2\left(\iint_B + \iint_{A \setminus B}\right) \frac{|\bp(x) - \bp(y)|^2}{|x-y|^{n+2t}} \delta(x)^{2\gamma} \mathrm{d}x\mathrm{d}y.
\end{aligned}
$$ 

First, by the gradient estimate of the Riesz potential in Lemma \ref{lem:riesz_diff}, 
it follows that  $|\nabla \boldsymbol{p}(x)| \le C(\Omega, s, \beta)\|f\|_{C^\beta(\overline{\Omega})}\, \delta(x)^{-s}$.
For simplicity, in what follows all constants hidden in the notation $\lesssim$ 
are understood to depend on the square of $C(\Omega,s,\beta) \|f\|_{C^\beta(\overline{\Omega})}$. 
Then, by the mean value theorem,
$$
\begin{aligned}
& ~\iint_{A \setminus B} \frac{|\bp(x) - \bp(y)|^2}{|x-y|^{n+2t}} \delta(x)^{2\gamma} \mathrm{d}x\mathrm{d}y  \\
\lesssim&  \int_{\Omega} \delta(x)^{-2s+2\gamma} \int_{|y-x| \leq \frac{1}{4}\delta(x)} |x-y|^{2-2t-n} \mathrm{d}y \mathrm{d}x \quad (\text{when } t < 1)\\
\lesssim&~ \frac{1}{1-t}\int_{\Omega} \delta(x)^{-2s + 2\gamma + 2 - 2t} \mathrm{d}x \lesssim \frac{1}{(1-t)(\frac32 - s + \gamma - t)},
\end{aligned}
$$ 
where \eqref{eq:delta-integral} is used in the last step, which requires $t < \tfrac{3}{2} - s + \gamma$.
Note that \eqref{eq:sobolev-regularity-all} implies $\bp \in H^{\frac{3}{2}-s-\varepsilon}(\Omega)$. Without loss of generality, we may assume $t \ge \gamma$ since $\frac32 - s > \frac12 -s \geq \gamma$. Hence, for the far-field integral, we have
$$
\begin{aligned}
\iint_{B} \frac{|\bp(x) - \bp(y)|^2}{|x-y|^{n+2t}} \delta(x)^{2\gamma} \mathrm{d}x\mathrm{d}y &\leq C \iint_B \frac{|\bp(x) - \bp(y)|^2}{|x-y|^{n+2(t-\gamma)}} \mathrm{d}x\mathrm{d}y \\
& \leq C |\bp|_{H^{t-\gamma}(\Omega)}^2 \leq C \frac{\|f\|_{L^2(\Omega)}^2}{\frac32 - s + \gamma - t}.
\end{aligned}
$$ 

\fbox{Case 2: $\gamma > s - \frac12$.} By using the gradient estimate of the Riesz potential \eqref{eq:riesz_grad_est} and \eqref{eq:delta-integral}, we have 
$$
|\bp|_{H^1_\gamma(\Omega)}^2 \lesssim \int_{\Omega} \delta(x)^{-2s + 2\gamma} \mathrm{d}x \lesssim \frac{1}{\frac12 -s + \gamma}.
$$

Finally, it remains to examine the condition for $|\nabla \bp|_{H_\gamma^{t-1}(\Omega)} < \infty$ for $t \in (1,2)$. 
First, by Lemma \ref{lem:riesz-gradient-Holder} (weighted H\"{o}lder estimate for $\nabla I_{2-2s}w$), 
we obtain
$$
\begin{aligned}
& ~\iint_{A \setminus B} \frac{|\nabla\bp(x) - \nabla\bp(y)|^2}{|x-y|^{n+2(t-1)}} \delta(x)^{2\gamma} \mathrm{d}x\mathrm{d}y  \\
\lesssim&  \int_{\Omega} \delta(x)^{-2s-2\beta+2\gamma} \int_{|y-x| \leq \frac{1}{4}\delta(x)} |x-y|^{2-2t + 2\beta-n} \mathrm{d}y \mathrm{d}x \quad (\text{when } t < \beta+1)\\
\lesssim&~ \frac{1}{\beta+1-t}\int_{\Omega} \delta(x)^{-2s + 2\gamma + 2 - 2t} \mathrm{d}x \lesssim \frac{1}{(\beta+1-t)(\frac32 - s + \gamma - t)},
\end{aligned}
$$ 
where the last step requires $t < \tfrac{3}{2} - s + \gamma$. 
$$
\begin{aligned}
\iint_{B} \frac{|\nabla\bp(x) - \nabla\bp(y)|^2}{|x-y|^{n+2(t-1)}} \delta(x)^{2\gamma} \mathrm{d}x\mathrm{d}y\leq&~C \iint_B \frac{|\nabla \bp(x)|^2 + |\nabla\bp(y)|^2}{|x-y|^{n+2(t-1)}} \delta(x)^{2\gamma} \mathrm{d}x\mathrm{d}y \\
\lesssim&  \int_{\Omega} \delta(x)^{-2s+2\gamma} \int_{|y-x| > \frac{1}{4}\delta(x)} |x-y|^{2-2t -n} \mathrm{d}y \mathrm{d}x \\
\lesssim&~ \frac{1}{(t-1)(\frac32 - s + \gamma - t)}.
\end{aligned} 
$$ 
Combining the two cases yields the estimate \eqref{eq:weighted-regularity-p}.
\end{proof}

\begin{remark}[differentiability order of $\bp$] \label{rmk:differentiability-p}
Unlike the differentiability of $u$ or $\bsigma$, the differentiability order of $\bp$ depends on the choice of the weight parameter $\gamma$. It can be less than $1$ when $\gamma \le s - \frac12$, or lies between $1$ and $2$ when $\gamma > s - \frac12$. In any case, the differentiability order never exceeds $2$, since $\beta + 1 < 3 - 2s < 2$.
\end{remark}

\section{Discretization and main results}\label{sc:discretization}
In this section, we introduce a local discontinuous Galerkin (LDG) discretization for the mixed formulation of the integral fractional Laplacian \eqref{eq:mixed_fractional_problem}. Inspired by the LDG scheme for second-order elliptic equations \cite{castillo2000priori}, the method is developed based on numerical fluxes. We refer to the survey \cite{arnold2002unified} for a unified analysis of various DG methods. In addition, in this section, we will also present the main results of this work and provide a discussion on the findings.

\subsection{LDG method}
Suppose $\{\mathcal{T}_h\}$ is a family of polygonal partitions of domain $\Omega$. Specifically, $\mathcal{T}_h$ represents a partition of $\Omega$ into polygon $T$, each with a diameter denoted by $h_T$. We define $\mathcal{F}_h$ as the set of all $(n-1)$-dimensional faces, $\mathcal{F}_h^0$ as the set of all interior faces, and $\mathcal{F}_h^\partial := \mathcal{F}_h \setminus \mathcal{F}_h^0$ as the set of boundary faces.

To obtain the a priori error estimates, some assumptions for the partition need to be proposed. We assume the family $\mathcal{T}_h$ to be shape regular, namely, 
\[
\sigma: = \sup\limits_{h>0} \max_{T\in \mathcal{T}_h} \frac{h_T}{\rho_T} < \infty,
\]
where $h_T = \text{diam}(T)$ and $\rho_T$ is the diameter of the largest ball contained in $T$. Furthermore, we assume that the partition is local quasi-uniform, which means $\exists\, \lambda > 0 $ such that
\[
h_T \leq \lambda h_{T'} \; \forall T, T' \in \mathcal{T}_h \text{ with } \overline{T} \cap \overline{T}' \neq \varnothing.
\]

To establish the weak formulation that defines the DG method, we multiply \eqref{eq:mixed_fractional_problem} by arbitrary smooth test functions $\btau, \bq$, and $v$, respectively, and apply integration by parts over each element $T \in \mathcal{T}_h$. This leads to the following system:  
\begin{equation} \label{eq:mixed_fractional_problem1}
	\left\{
	\begin{aligned}
		\int_{T} I_{2-2s} \bsigma \cdot \btau\,\mathrm{d}x - \int_{T} \bp \cdot \btau\,\mathrm{d}x &= 0,\\			
		\int_{T} \bsigma \cdot \bq\,\mathrm{d}x + \int_{T} u\,\div \bq\,\mathrm{d}x - \int_{\partial T} u\,\bq \cdot \bn\,\mathrm{d}s &= 0,\\		
		\int_{T} \bp \cdot \nabla v\,\mathrm{d}x - \int_{\partial T} \bp \cdot \bn\,v\,\mathrm{d}s &= \int_{T} f v\,\mathrm{d}x,
	\end{aligned}
	\right.	
\end{equation}
where $\bn$ denotes the outward unit normal vector on $\partial T$. 
These equations are well-defined for arbitrary functions $(\bsigma, \bp, u)$ and $(\btau, \bq, v)$ in $\bSigma \times \bQ \times V$, where  
\[
\begin{aligned}
	\bSigma &:= \left\{\btau \in \boldsymbol{L}^2(\mathbb{R}^n) : \btau = 0 \text{ in }  \Omega^c \right\},\\
	\bQ &:= \left\{\bq \in \boldsymbol{L}^2(\mathbb{R}^n) : \bq|_T \in \boldsymbol{H}^1(T), \; \forall\, T \in \mathcal{T}_h \right\},\\
	V &:= \left\{v \in L^2(\mathbb{R}^n) : v = 0 \text{ in } \Omega^c, v|_T \in H^1(T), \; \forall\, T \in \mathcal{T}_h \right\}.
\end{aligned}
\]  

Noting that \eqref{eq:mixed_fractional_problem1} involves only the information of $\bp$ within $\Omega$, we therefore seek a scalar approximation $u_h$ for $u$, along with vector approximations $(\bsigma_h, \bp_h)$ for $(\bsigma, \bp|_{\Omega})$, in the finite element space $\bSigma_h \times \bQ_h \times V_h \subset \bSigma \times \bQ|_{\Omega} \times V$, where  
\begin{equation} \label{eq:DG-spaces}
	\begin{aligned}
		\bSigma_h & := \{\btau \in \boldsymbol{L}^2(\mathbb{R}^n) : \btau = 0 \text{ in }  \Omega^c, \btau\vert_T \in \bSigma(T), \forall\, T \in \mathcal{T}_h \},\\
		\bQ_h & := \{\bq \in \boldsymbol{L}^2(\Omega): \bq\vert_T \in \bQ(T), \forall\, T \in \mathcal{T}_h \},\\
		V_h & := \{v \in L^2(\mathbb{R}^n) : v = 0 \text{ in } \Omega^c, v\vert_T \in V(T), \forall\, T \in \mathcal{T}_h \},
	\end{aligned}
\end{equation}
respectively. Here, $\bSigma(T)$, $\bQ(T)$, and $V(T)$ denote local finite element spaces, which typically consist of polynomials. The approximate solution $(\bsigma_h, \bp_h, u_h)$ is then defined using the weak formulation described above. Since the test functions belong to discontinuous finite element spaces, the boundary values on each element are represented by the \emph{numerical fluxes} $\widehat{u}_h$ and $\widehat{\bp}_h$, which are introduced to ensure the stability of the method and to improve its accuracy.


\paragraph{LDG numerical fluxes} We first introduce some classical notation. Let $T^+$ and $T^-$ be two adjacent elements of $\mathcal{T}_h$, and let $F = \partial T^+ \cap \partial T^-$, which is assumed to have a nonzero $(n-1)$-dimensional measure. The outward unit normals on $F$ are denoted by $\bn^+$ and $\bn^-$, corresponding to $T^+$ and $T^-$, respectively.
Let $(\bq, v)$ be functions that are smooth within each element $T^\pm$. We denote by $(\bq^\pm, v^\pm)$ the traces of $(\bq, v)$ on $F$ from the interiors of $T^+$ and $T^-$, respectively. Using these, we define the mean values $\{\cdot\}$ (or $\vavg{\cdot}$) and the jumps $[\cdot]$ (or $\llbracket \cdot \rrbracket$) on $F \in \mathcal{F}_h^0$ as follows:
\[
\begin{alignedat}{2}
  &\{v\} := \tfrac12(v^+ + v^-), \quad &\vavg{\bq} :=& \tfrac12(\bq^+ + \bq^-),\\
  &\llbracket v\rrbracket := v^+\bn^+ + v^-\bn^-, \quad &[\bq] :=& \bq^+\!\cdot\!\bn^+ + \bq^-\!\cdot\!\bn^-.
\end{alignedat}
\]
And for the face $F\in \mathcal{F}_h^\partial $, define
\[
\{v\} := v, \quad \vavg{\bq} := \bq, \quad 
\llbracket v\rrbracket := v\bn, \quad [\bq] := \bq\!\cdot\!\bn.
\]

Following \cite{castillo2000priori}, we present the general form of LDG numerical fluxes as follows. 
	\begin{equation} \label{eq:numer-flux}
	\left\{
	\begin{aligned}
		\widehat{\bp}_h\vert_{F} &= \vavg{\bp_h} - \bbeta [\bp_h]- C_{11}\llbracket u_h \rrbracket, \\
		\widehat{u}_h\vert_{F} &= \{u_h\} + \bbeta \cdot \llbracket u_h \rrbracket,
	\end{aligned} 
	\qquad \forall F \in \mathcal{F}_h^0.
	\right.
	\end{equation}
where the parameters $C_{11}$, $\bbeta$ depend on $F$.  
Since $u$ vanishes outside the domain, and $\bp$ has no prescribed boundary condition on $\partial\Omega$ (thus being {\it natural} in a certain sense), we set:  
\begin{equation} \label{eq:numer-flux-boundary}
\left\{
\begin{aligned}
	\widehat{\bp}_h\vert_{F} &= \bp_h - C_{11} u_h \bn, \\
	\widehat{u}_h\vert_{F} &= 0, 
\end{aligned}
\right.
\qquad \forall F \in \mathcal{F}_h^\partial.
\end{equation}

\paragraph{Stabilization for vector variables} In the present formulation, two vector-valued variables, $\bsigma$ and $\bp$, are introduced, which distinguishes the method from the standard formulation of second-order elliptic problems. At the discrete level, these variables are approximated in possibly different finite element spaces, $\bSigma_h$ and $\bQ_h$. This discrepancy reflects the distinct differentiability properties of the two variables, as discussed earlier in Remark \ref{rmk:differentiability-u-sigma} and Remark \ref{rmk:differentiability-p}. To accommodate such a difference, a stabilization mechanism is incorporated by adding the term
\begin{equation}\label{eq:stabilization}
  \sum_{T \in \mathcal{T}_h} C_{\rm s} \int_{T} (\bI - \bPi_{\bSigma}^0)\bp_h \cdot (\bI - \bPi_{\bSigma}^0)\bq_h \mathrm{d}x,
\end{equation}
where $C_{\rm s} > 0$ is a stabilization parameter, and $\bPi_{\bSigma}^0$ denotes the $L^2$-projection onto $\bSigma_h$. Note that when $\bQ(T) = \bSigma(T)$, the above stabilization term automatically vanishes.

\paragraph{LDG method} 
Using the above notation in \eqref{eq:mixed_fractional_problem1} yields the following LDG method: 
Find $(\bsigma_h, \bp_h, u_h) \in \bSigma_h \times \bQ_h \times V_h$ such that, 
for all $(\btau_h, \bq_h, v_h) \in \bSigma_h \times \bQ_h \times V_h$,
\begin{equation}\label{eq:mixed_scheme}
\left\{
\begin{aligned}
\int_{\Omega} I_{2-2s} \bsigma_h \cdot \btau_h \, \mathrm{d}x
  - \int_{\Omega} \bp_h \cdot \btau_h \, \mathrm{d}x 
  &= 0,\\[1mm]
\int_{\Omega} \bsigma_h \cdot \bq_h \, \mathrm{d}x
  + \sum_{T \in \mathcal{T}_h} C_{\rm s} 
    \int_{T} (\bI - \bPi_{\bSigma}^0)\bp_h \cdot (\bI - \bPi_{\bSigma}^0)\bq_h \, \mathrm{d}x
  &\\[-1mm]
  + \sum_{T \in \mathcal{T}_h} \int_{T} u_h \div \bq_h \, \mathrm{d}x
  - \sum_{F \in \mathcal{F}_h} \int_{F} \widehat{u}_h [\bq_h] \, \mathrm{d}s
  &= 0,\\[1mm]
\sum_{T \in \mathcal{T}_h} \int_{T} \bp_h \cdot \nabla v_h \, \mathrm{d}x
  - \sum_{F \in \mathcal{F}_h} \int_{F} \widehat{\bp}_h \cdot \llbracket v_h \rrbracket \, \mathrm{d}s
  &= \int_{\Omega} f v_h \, \mathrm{d}x,
\end{aligned}
\right.
\end{equation}
where the numerical fluxes $\widehat{u}_h$ and $\widehat{\bp}_h$ are given in \eqref{eq:numer-flux} and \eqref{eq:numer-flux-boundary}.


\begin{remark}[nonlocal coupling]
	One of the key ideas behind the LDG method for second-order problems is to introduce auxiliary variables, rewriting the problem into a lower-order system. Local numerical fluxes are then defined to maintain local conservativity. The scheme \eqref{eq:mixed_scheme} follow this approach. However, due to the introduction of the Riesz potential, the coupling of $\bsigma_h$ is inherently nonlocal, which corresponds to the nature of the fractional Laplacian.
	On the other hand, since \(\bsigma_h, \btau_h \in \bSigma_h\) are supported in \(\Omega\), we have  
\[
\int_{\Omega}(I_{2-2s}\bsigma_h)(x) \cdot \btau_h(x)\mathrm{d}x 
= \frac{1}{\gamma(2-2s)}\int_{\Omega} \int_\Omega \frac{\bsigma_h(y) \cdot \btau_h(x)}{|x-y|^{n-(2-2s)}} \mathrm{d}y\mathrm{d}x.
\]
Thus, computation only needs to be performed within the domain $\Omega$, which is notably different from the primal formulation of the fractional Laplacian, where integrals over the entire $\mathbb{R}^n$ must be considered (see \cite{acosta2017short}). 
\end{remark}

\subsection{The classical mixed setting}
The study of DG method is greatly facilitated if we recast its formulation in a classical mixed finite element setting. Above and throughout, we use the notation
\[
(v,w)_{\mathcal{T}_h} = \sum_{T\in \mathcal{T}_h} (v,w)_T\qquad\text{and}\qquad \langle v,w\rangle_{\partial\mathcal{T}_h}=\sum_{T\in \mathcal{T}_h}\langle v,w\rangle_{\partial T}, 
\]
where we write $(u,v)_D = \int_D uv\mathrm{d}x$ whenever $D$ is a domain of $\mathbb{R}^n$, and $\langle u,v\rangle_D = \int_D uv\mathrm{d}x$ whenever $D$ is a domain of $\mathbb{R}^{n-1}$. For the vector functions $\boldsymbol{v}$ and $\boldsymbol{w}$, the notation is similarly defined with the integrand being the dot product $\boldsymbol{v}\cdot \bw$.
In particular, for notational simplicity, we denote $(\cdot,\cdot):=(\cdot,\cdot)_{\Omega}$.

We take the numerical fluxes $\widehat{u}_h$ and $\widehat{\bp}_h$ as defined in \eqref{eq:numer-flux} and \eqref{eq:numer-flux-boundary} substitute them into the scheme \eqref{eq:mixed_scheme}. 
The resulting LDG scheme reads: Find $(\bsigma_h, \bp_h, u_h) \in \bSigma_h \times \bQ_h \times V_h$ such that
\begin{equation}\label{eq:short_LDG_scheme}
\left\{
\begin{array}{r@{\;}c@{\;}l @{\;}c@{\;}l @{\;}c@{\;}l @{\quad}l}
  a(\bsigma_h, \btau_h) & -&  (\bp_h, \btau_h) &   &              & = & 0       & \forall \btau_h \in \bSigma_h,\\
  (\bsigma_h, \bq_h)    & + & a_{\rm s}(\bp_h, \bq_h)  & + & b(u_h, \bq_h) & = & 0       & \forall \bq_h \in \bQ_h,\\
                         &  - &  b(v_h, \bp_h) & + & c(u_h, v_h)   & = & (f, v_h) & \forall v_h \in V_h.
\end{array}
\right.
\end{equation}
where
\begin{equation}\label{eq:definition_abc}
	\begin{aligned}
		a(\bsigma_h,\btau_h)& := \int_{\Omega}I_{2-2s}\bsigma_h \cdot \btau_h \mathrm{d}x, \\
		a_{\rm s}(\bp_h, \bq_h) &:= (C_{\rm s} ( \bI - \bPi_\bSigma^0)\bp_h, ( \bI - \bPi_\bSigma^0)\bq_h)_{\mathcal{T}_h}, \\
		b(u_h,\bq_h) &:= ( u_h, \div \bq_h)_{\mathcal{T}_h} -\int_{\mathcal{F}_h^0} (\{u_h\}+\bbeta\cdot \llbracket u_h \rrbracket) [ \bq_h ]\mathrm{d}s\\
		& = -(\nabla u_h,\, \bq_h)_{\mathcal{T}_h} + \int_{\mathcal{F}^0_h} (\vavg{\bq_h} - \bbeta[\bq_h]) \cdot \llbracket u_h\rrbracket\mathrm{d}s + \int_{\partial \Omega}u_h \bq_h\cdot \bn \mathrm{d}s, \\
		c(u_h,v_h) &:= \int_{\mathcal{F}_h} C_{11} \llbracket u_h \rrbracket\cdot \llbracket v_h\rrbracket \mathrm{d}s.
	\end{aligned}
\end{equation}

\begin{lemma}[well-posedness of the LDG method]
For the LDG method \eqref{eq:short_LDG_scheme}, assume that the finite element spaces satisfy $\nabla V(T) \subset \bQ(T)$. 
Moreover, let the coefficients in the numerical fluxes \eqref{eq:numer-flux} satisfy $C_{11} > 0$, 
and let the stabilization parameter in \eqref{eq:stabilization} satisfy $C_s > 0$. 
Then, the LDG method admits a unique solution $(\bsigma_h, \bp_h, u_h) \in \bSigma_h \times \bQ_h \times V_h$.
\end{lemma}

\begin{proof}
	Due to the linearity and finite dimensionality of the problem, it suffices to show that the only solution to \eqref{eq:short_LDG_scheme} with $f = 0$ is $\bsigma_h = 0$, $\bp_h = 0$, and $u_h = 0$. By choosing the test functions $\btau_h = \bsigma_h$, $\bq_h = \bp_h$, and $v_h = u_h$, and summing the three equations, we obtain
	\[
	a(\bsigma_h,\bsigma_h) + a_{\rm s}(\bp_h, \bp_h) + c(u_h,u_h) = 0.
	\]
	Noting that each term in the above equation is nonnegative, it follows that they must all vanish individually.
	
	Using the Fourier transform of Riesz potential \cite[Lemma 1, Chapter V]{stein1970singular}, we have 
	$$
	0 = a(\bsigma_h, \bsigma_h) = (I_{2-2s}\tilde{\bsigma}_h, \tilde{\bsigma}_h)_{L^2(\mathbb{R}^n)} = (2\pi)^{2s-2} \int_{\mathbb{R}^n}|\xi|^{2s -2}|\mathcal{F}[\tilde{\bsigma}_h](\xi)|^2\mathrm{d}\xi,
	$$ 
	we conclude that $\bsigma_h = 0$. Consequently, the first equation reduces to
	$(\bp_h,\btau_h) = 0, \forall \btau_h \in \bSigma_h$, which implies $\bPi_{\bSigma}^0 \bp_h = 0$.
Moreover, since $a_{\rm s}(\bp_h, \bp_h) = 0$ and $C_s > 0$, we also have $(\bI - \bPi_{\bSigma}^0)\bp_h = 0$.
Combining these results yields $\bp_h = 0$.
	
Subsequently, since $C_{11} > 0$ and $c(u_h, u_h) = 0$, it follows that $\llbracket u_h \rrbracket = 0$ on $F \in \mathcal{F}_h^0$ and $u_h = 0$ on $F \in \mathcal{F}_h^\partial$. At this point, the second equation reduces to $b(u_h, \bq_h) = 0,  \forall \bq_h \in \bQ_h$. Therefore, we have  
\[
0 = b(u_h, \bq_h) = -(\nabla u_h, \bq_h)_{\mathcal{T}_h} \quad \forall \bq_h \in \bQ_h.  
\]  
Since $\nabla V(T) \subset \bQ(T)$, it follows that $\nabla u_h = 0$ on each $T \in \mathcal{T}_h$, which implies that $u_h$ is piecewise constant. Given that $\llbracket u_h \rrbracket = 0$ on $\mathcal{F}_h^0$ and $u_h = 0$ on the boundary, we conclude that $u_h = 0$, which completes the proof.
\end{proof}

\begin{remark}[reduced linear system] \label{rmk:LDG}
As in the LDG methods for second-order elliptic problems, the flux variables in the fractional Laplacian LDG scheme \eqref{eq:short_LDG_scheme} can be locally eliminated, so that the dimension of the resulting algebraic system reduces to $\dim V_h$. To illustrate this, consider the case $\boldsymbol{\Sigma}(T)\subset \boldsymbol{Q}(T)$. Then the  $L^2$-orthogonal decomposition holds $\bQ_h = \bSigma_h \oplus (\bI - \bPi_{\bSigma}^0) \bQ_h$.

Let $S_h$, $P_h^{0}$, $P_h^{\perp}$, and $U_h$ denote the vector representations of $\bsigma_h$, $\bPi_{\bSigma}^0 \bp_h$, $(\bI - \bPi_{\bSigma}^0)\bp_h$, and $u_h$, respectively. The algebraic system associated with \eqref{eq:short_LDG_scheme} then takes the form
$$
\begin{pmatrix}
A & -M_0 & 0 & 0 \\
M_0 & 0 & 0 & B_0^T \\
0 & 0 & M_{{\rm s},\perp} & B_{\perp}^T \\
0 & -B_0 & - B_{\perp} & C
\end{pmatrix}
\begin{pmatrix}
S_h \\ P_h^0 \\ P_h^\perp \\ U_h
\end{pmatrix} = 
\begin{pmatrix}
0 \\ 0 \\ 0 \\ F
\end{pmatrix}. 
$$ 
Here, $A$ and $C$ denote the matrix representations of $a(\cdot,\cdot)$ and $c(\cdot,\cdot)$, respectively. The matrices $M_{0}$ and $M_{\mathrm{s},\perp}$ are the mass matrix on $\bSigma_h$ and the matrix representation of $a_{\mathrm{s}}(\cdot,\cdot)$ restricted to $(\bI - \bPi_{\bSigma}^0) \bQ_h$, respectively --- both have the block-diagonal structure. Moreover, $B_{0}$ and $B_{\perp}$ denote the matrix representations of $b(\cdot,\cdot)$ on $V_h\times \boldsymbol{\Sigma}_h$ and $V_h\times (\bI - \bPi_{\bSigma}^0) \bQ_h$, respectively.
Owing to the block structure of $M_{0}$ and $M_{\mathrm{s},\perp}$, the above linear system can be reduced to a single equation for $U_h$ of the form 
\begin{equation} \label{eq:reduced-system}
A_{\rm LDG} U_h = F, \quad \text{where} \quad A_{\rm LDG} := B_0M_0^{-1} A M_0^{-1} B_0^T + B_\perp M_{{\rm s}, \perp}^{-1} B_\perp^T + C.
\end{equation}
\end{remark}

\subsection{Main results based on regularity theory} \label{subsec:main}
We first introduce a naturally arising {\it energy} seminorm that emerges in error estimation for the LDG method: $\forall \tau\in \bSigma+\bSigma_h, v\in V+ V_h$, define
\begin{equation} \label{eq:energy-norm}
|(\btau, v)|^2_{\mathcal{A}} :=(I_{2-2s}\btau,\btau) + c(v,v).
\end{equation}
 Since $\bsigma=\nabla u$, Lemma \ref{lem:riesz_grad} (equivalent form
 of $\widetilde{H}^s$-norm) leads to $(I_{2-2s}\bsigma,\bsigma)\eqsim \|u\|_{\widetilde{H}^s(\Omega)}^2$, which justifies calling it the energy seminorm.

\subsubsection{Results on quasi-uniform meshes}
By the Sobolev regularity result in \eqref{eq:sobolev-regularity-all}, $\bsigma \in \boldsymbol{H}^{s-\frac12-\varepsilon}(\Omega)$ and $\bp \in \boldsymbol{H}^{\frac{3}{2}-s-\varepsilon}(\Omega)$. In this case, it is sufficient to choose $\bSigma(T)=\bQ(T)=\boldsymbol{\mathcal{P}}_{0}(T)$, and no stabilization term such as \eqref{eq:stabilization} is required. 
Equivalently, one may set $C_{\rm s}=0$.
	
\begin{theorem}[error estimate on quasi-uniform meshes]\label{tm:uniform_estimate}
Let $\Omega\subset \mathbb{R}^n$ be a bounded Lipschitz domain, $s\in(\frac12,1)$, $f\in L^2(\Omega)$, and 
$(\bsigma,\bp,u)\in \bSigma\times \bQ \times V$ be the solution of (\ref{eq:mixed_fractional_problem}), $(\bsigma_h,\bp_h,u_h)\in \bSigma_h\times \bQ_h\times V_h$ be the approximation solution of \eqref{eq:short_LDG_scheme} on shape-regular triangulation, with 
$$
\bSigma(T) = \boldsymbol{\mathcal{P}}_0(T), \quad \bQ(T) = \boldsymbol{\mathcal{P}}_0(T), \quad V(T) = \mathcal{P}_1(T).
$$
Then, under the condition $C_{11} \eqsim h^{1-2s}$, we have
        \begin{subequations}
	\begin{align}
		|(\bsigma-\bsigma_h, u-u_h)|_{\mathcal{A}} & \leq C(\Omega,n,s,\sigma) 
		h^{\frac{1}{2}} |\log h|^{\frac12} \|f\|_{L^2(\Omega)}, \label{eq:energy-result-uniform} \\
		\|u - u_h\|_{L^2(\Omega)} & \leq C(\Omega,n,s,\sigma)  h |\log h| \|f\|_{L^2(\Omega)}. \label{eq:L2-result-uniform}
		\end{align}
	\end{subequations}
\end{theorem}

\subsubsection{Results on graded meshes}
Based on the weighted Sobolev regularity result (Theorem \ref{thm:weighted_sobolev_estimate} and Theorem \ref{thm:weighted_sobolev-p}), we could improve the convergence rates by employing the following \emph{graded meshes}, which is widely used for problems with solution singularities: A shape regular and local quasi-uniform partition $\mathcal{T}_h$ is called a graded mesh, if there is a number $\mu \geq 1$ such that for any $T \in\mathcal{T}_h$,
\begin{equation}\label{eq:graded_mesh}
		h_T \eqsim
		\left\{
		\begin{aligned}
			&h^\mu \quad &\text{ if } T\cap \partial \Omega \neq \varnothing,\\
			& h\text{dist}(T,\partial \Omega)^{\frac{\mu-1}{\mu}}\quad & \text{ if } T\cap \partial \Omega = \varnothing.
		\end{aligned}
	\right.
\end{equation}

According to Remark \ref{rmk:LDG}, the dimension of the resulting algebraic system equals $\dim V_h$, and the total number of degrees of freedom is therefore set to $N := \dim V_h$. The construction \eqref{eq:graded_mesh} yields \cite{babuvska1979direct}:
\begin{equation}\label{eq:graded_number}
	N \eqsim \#\mathcal{T}_h \eqsim 
	\begin{cases}
		h^{-n} & \text{if } \mu <  \frac{n}{n-1},\\
		h^{-n} |\log h|  & \text{if } \mu = \frac{n}{n-1},\\
		h^{(1-n)\mu} & \text{if } \mu > \frac{n}{n-1}.
	\end{cases}
\end{equation}
The relationship restricts the graded parameter $\mu \leq \frac{n}{n-1}$, since the larger parameter will not be benefit for improving the convergence rate, see \cite{acosta2017fractional,borthagaray2021local,han2022monotone}.
\begin{theorem}[error estimate on graded meshes]\label{tm:graded_estimate}
Let $\Omega \subset \mathbb{R}^n$ be a bounded Lipschitz domain satisfying the exterior ball condition.
Let $\beta>0$ and a family of graded triangulations 
$\{\mathcal{T}_h\}$ with grading parameter $\mu$ be chosen such that
	\begin{equation}\label{eq:parameter_condition}
	\beta\geq 
	\begin{cases}
		2-2s & \text{if } n=1,\\
		\frac{n}{2(n-1)} -s & \text{if } n\geq2,
	\end{cases}
	\quad \text{and} \quad
	\mu = \begin{cases}
		4-2s & \text{if } n=1,\\
		\frac{n}{n-1} & \text{if }n\geq 2.
	\end{cases}
	\end{equation}
   For $f \in C^\beta(\overline{\Omega})$, let $(\bsigma,\bp,u)\in \bSigma\times \bQ \times V$ be the solution of (\ref{eq:mixed_fractional_problem}).  Take the local spaces 
    $$
	\bSigma(T) = \boldsymbol{\mathcal{P}}_0(T), \quad \bQ(T) = \boldsymbol{\mathcal{P}}_1(T), \quad V(T) = \mathcal{P}_1(T), 
    $$
   and choose the parameters as
   \begin{equation} \label{eq:parameter-C}
   \begin{aligned}
   C_{11}|_F &\eqsim h_{T}^{1-2s} \text{ for } F \subset \partial T, \\
   C_{\rm s}|_T & \eqsim 
   \begin{cases}
   h_T^{2-2s} & \text{if } n=1,\\
   h_T^{2\theta} & \text{if } n \geq 2, \theta \geq \max\{\frac{n}{2(n-1)}, s\} - 1.
   \end{cases}
   \end{aligned}
   \end{equation}
   With these choices, let $(\bsigma_h,\bp_h,u_h)\in \bSigma_h\times \bQ_h\times V_h$ be the approximation solution of \eqref{eq:short_LDG_scheme} on $\mathcal{T}_h$.
   Then, we have
    \begin{subequations} \label{eq:result-graded-h}
\begin{align} 
	|(\bsigma-\bsigma_h,u-u_h)|_{\mathcal{A}} & \lesssim
	\begin{cases}
	h^{2-s} |\log h|  \|f\|_{C^\beta(\overline{\Omega})} & \text{if } n=1, \\
	h^{\frac{n}{2(n-1)}} |\log h| \|f\|_{C^\beta(\overline{\Omega})} & \text{if } n\geq 2,
	\end{cases}  \label{eq:energy-result-graded} \\
	\|u - u_h\|_{L^2(\Omega)} & \lesssim
	\begin{cases}
	h^{\frac{5}{2}-s} |\log h|^{\frac32}  \|f\|_{C^\beta(\overline{\Omega})} & \text{if } n=1, \\
	h^{\frac{n}{2(n-1)} + \frac12} |\log h|^{\frac32} \|f\|_{C^\beta(\overline{\Omega})} & \text{if } n\geq 2.
	\end{cases} 
	 \label{eq:L2-result-graded}
\end{align}
\end{subequations}
Here, the hidden constant $C = C(\Omega, n, s, \beta, \sigma)$.
In terms of number of degrees of freedom $N$, the estimates \eqref{eq:result-graded-h} read
    \begin{subequations} \label{eq:result-graded-N}
\begin{align} 
	|(\bsigma-\bsigma_h,u-u_h)|_{\mathcal{A}} & \lesssim
	\begin{cases}
	N^{-(2-s)} (\log N) \|f\|_{C^\beta(\overline{\Omega})} & \text{if } n=1, \\
	N^{-\frac{1}{2(n-1)}} (\log N)^{\frac{1}{2(n-1)} + 1} \|f\|_{C^\beta(\overline{\Omega})} & \text{if } n\geq 2,
	\end{cases}  \label{eq:energy-result-graded2}  \\
	\|u - u_h\|_{L^2(\Omega)} & \lesssim 
	\begin{cases}
	N^{-(\frac52-s)} (\log N)^{\frac32} \|f\|_{C^\beta(\overline{\Omega})} & \text{if } n=1, \\
	N^{-\frac{1}{2(n-1)} - \frac{1}{2n}} (\log N)^{\frac{1}{2(n-1)} + \frac32} \|f\|_{C^\beta(\overline{\Omega})} & \text{if } n\geq 2.
	\end{cases} 
	 \label{eq:L2-result-graded2}
\end{align}
\end{subequations}
\end{theorem}

\section{Proofs}\label{sc:triangle}
This section presents the proofs of the two main theorems. 

\subsection{Interpolations}
We begin by recalling several local $L^{2}$-projection estimates onto the
polynomial spaces $\mathcal{P}_{k}(T)$, $k=0,1$.  
Let $\Pi_{k}^{0}$ denote the corresponding $L^{2}$-projection. 
We begin by recalling the following classical result:
\begin{subequations} \label{eq:Pi-uniform}
\begin{align} 
\|w - \Pi_0^0 w\|_{L^2(T)} 
  &\leq C(n,\sigma, t) h_T^{\min\{t,1\}}\|w\|_{H^{t}(T)} \qquad t \ge 0, \label{eq:Pi-T} \\ 
\|w - \Pi_0^0 w\|_{L^2(\partial T)} 
  &\leq C(n,\sigma, t) h_T^{\min\{t,1\} - \frac12}\|w\|_{H^{t}(T)} ~~ t > \tfrac{1}{2}. \label{eq:Pi-partial-T}
\end{align}
\end{subequations}

We now present local interpolation estimates in the weighted Sobolev space \eqref{eq:weighted_norm}, which only needs to be carried out elementwise on each $T$.
By the weighted fractional Poincaré inequality (see \cite[Proposition 4.8]{acosta2017fractional}), it follows that
\begin{equation} \label{eq:Pi-0-weighted}
  \|w - \Pi_{0}^0 w\|_{L^{2}(T)} \le C(n, \sigma, \ell, \gamma) h_{T}^{\ell-\gamma} |w|_{H_{\gamma}^{\ell}(T)} \quad \forall \ell \in (0,1], \gamma \in [0, \ell).
\end{equation}
Further, using trace theorem, scaling argument and Bramble-Hilbert lemma, the following local interpolation estimate holds for $\Pi_{1}^{0}$ (cf. \cite{acosta2017fractional,borthagaray2019weighted,borthagaray2021local}): 
\begin{equation} \label{eq:Pi-1-weighted}
\|w - \Pi_1^0 w\|_{L^2(\partial T)} \leq C(n,\sigma, t, \gamma) h_T^{t-\frac12 - \gamma} |w|_{H^t_\gamma(T)}  \quad  \forall t \leq 2, \gamma \in [0, t - \tfrac12).
\end{equation}     

Since $V_h$ is the discontinuous piecewise $\mathcal{P}_1$ finite element space, we directly employ the Scott--Zhang interpolation operator $\Pi^{\rm SZ}$ as discussed in \cite[Section 4]{acosta2017fractional}, whose range is the conforming space $\mathbb{V}_h^{\rm c} := \{ v\in C^0(\Omega) : v|_T \in \mathcal{P}_1(T) \ \text{for all } T\in \mathcal{T}_h \}$.
For completeness, we briefly outline the interpolation estimates; see 
\cite{acosta2017fractional, borthagaray2021local} and the references therein for details. 
First, the fractional Hardy inequality allows us to replace the control of 
$\|\cdot\|_{\widetilde{H}^s(\Omega)}$ by that of $\|\cdot\|_{H^s(\Omega)}$. 
The latter can be localized by means of patchwise estimates, reducing to 
the collection of sets $\{T \times S_T\}_{T \in \mathcal{T}_h}$ together with local zero-order 
contributions ($S_T$ denotes the first ring of $T$). Establishing suitable local interpolation bounds on each patch then leads to the 
following interpolation estimates on quasi-uniform meshes:
\begin{equation} \label{eq:SZ-uniform}
\|w - \Pi^{\rm SZ}w\|_{\widetilde{H}^s(\Omega)} \leq C(\Omega, n,s,\sigma,t) h^{t-s}|w|_{H^t(\Omega)}, \quad \frac12<s<1 \text{ and } t\in(s,2]. 
\end{equation}

On the graded meshes specified in \eqref{eq:graded_mesh}, the local interpolation bounds need to be treated in two separate cases: ($S_{S_T}$ denotes the second ring of $T$)
$$
\begin{aligned}
S_{S_T} \text{ touches boundary}: \quad h_T^{t-s-\gamma}
  |u|_{H_\gamma^{t}(S_{S_T})} &\eqsim h^{\mu (t-s-\gamma)} |u|_{H_\gamma^{t}(S_{S_T})}, \\
\text{otherwise}: \qquad h_T^{\,t-s}\,|u|_{H^{t}(S_{S_T})} &\eqsim h^{t-s}\delta(T,\partial\Omega)^{(t-s)\frac{\mu-1}{\mu}}
        |u|_{H^{t}(S_{S_T})}.
\end{aligned}
$$
Balancing the contributions by choosing $\gamma = (t-s)\frac{\mu-1}{\mu}$, we obtain
\begin{equation} \label{eq:SZ-graded}
\|w - \Pi^{\rm SZ}w\|_{\widetilde{H}^s(\Omega)} \leq C(\Omega, n,s,\sigma,t)h^{t-s} |w|_{H_\gamma^t(\Omega)}, \quad \frac12<s<1 \text{ and } t\in(s,2]. 
\end{equation}
It is noted that the parameter $t$ must satisfy the conditions stated in Theorem \ref{thm:weighted_sobolev_estimate} (weighted Sobolev regularity of $u$), i.e., $t < s + \frac12 + \gamma$, which leads to $t-s < \frac{\mu}{2}$. Combined with the requirement $\mu \leq \frac{n}{n-1}$ in~\eqref{eq:graded_number}, an optimal choice of parameters is then obtained in the subsequent proofs.

\subsection{Error equations}
Firstly, we consider the error equations. For $\bsigma = \nabla u$ and $\bp = I_{2-2s}\nabla u$, we have the following consistency:
\begin{equation}\label{eq:fracDG_consistency}
\left\{
\begin{array}{r@{\;}c@{\;}l @{\;}c@{\;}l @{\;}c@{\;}l @{\quad}l}
  a(\bsigma,\btau_h) & - & (\bp,\btau_h) &   &             & = & 0        & \forall \btau_h \in \bSigma_h,\\
  (\bsigma,\bq_h)    &   &               & + & b(u,\bq_h)  & = & 0        & \forall \bq_h \in \bQ_h,\\
                     &   - & b(v_h,\bp)   & + & c(u,v_h)    & = & (f,v_h)  & \forall v_h \in V_h .
\end{array}
\right.
\end{equation}
Subtracting the LDG scheme \eqref{eq:short_LDG_scheme} from the consistency relation \eqref{eq:fracDG_consistency} yields the error equations:
\begin{equation}\label{eq:error_equation}
\left\{
\begin{array}{r@{\;}c@{\;}l @{\;}c@{\;}l @{\;}c@{\;}l @{\quad}l}
  a(\bsigma-\bsigma_h,\btau_h) & - & (\bp-\bp_h,\btau_h) &   &                  & = & 0 & \forall \btau_h \in \bSigma_h,\\
  (\bsigma-\bsigma_h,\bq_h)    & - & a_{\rm s}(\bp_h, \bq_h)                      & + & b(u-u_h,\bq_h)  & = & 0 & \forall \bq_h \in \bQ_h,\\
                               & - & b(v_h,\bp-\bp_h)     & + & c(u-u_h,v_h)    & = & 0 & \forall v_h \in V_h .
\end{array}
\right.
\end{equation}

In the following, we take $\bPi\bsigma:=\nabla\Pi^{\rm SZ}u \in \bSigma_h $, $\Pi u :=\Pi^{\rm SZ}u \in V_h $ and $\bPi\bp:=\bPi_{\bQ}^{0} \bp \in \bQ_h$, where $\bPi_{\bQ}^{0}$ denotes the $L^2$-projection onto $\bQ_h$. We also introduce the following abbreviated notation
\[
\begin{aligned}
	\bE_\bsigma &:= \bPi\bsigma - \bsigma_h, &  \bE_\bp &:= \bPi \bp - \bp_h, & E_u &:= \Pi u-u_h,\\
	\be_{\bsigma}&: = \bPi\bsigma - \bsigma, & \be_\bp&:= \bPi\bp- \bp, & e_u &:=\Pi u - u.
\end{aligned}
\]
From the error equations \eqref{eq:error_equation}, we obtain the following system: $\forall (\btau_h, \bq_h, v_h) \in \bSigma_h \times \bQ_h \times V_h$, 
\begin{subequations}\label{eq:energy_primal}
	\begin{align}
		a(\bE_\bsigma,\btau_h)-(\bE_\bp,\btau_h) &= a(\be_\bsigma,\btau_h)- (\be_\bp,\btau_h), \label{eq:energy_primal_1}\\
		(\bE_\bsigma ,\bq_h) + a_{\rm s}(\bE_\bp, \bq_h) + b(E_u ,\bq_h) &= (\be_\bsigma,\bq_h) + a_{\rm s}(\bPi_{\bQ}^0 \bp, \bq_h) + b(e_u ,\bq_h),\label{eq:energy_primal_2}\\
		-b(v_h,\bE_\bp) + c(E_u, v_h) &= -b(v_h,\be_\bp) + c(e_u, v_h) \label{eq:energy_primal_3}.
	\end{align}
\end{subequations} 

\subsection{Energy error estimates}
\begin{lemma}[energy error estimates] \label{lm:energy_estimate_triangle_LDG}
Under the condition in Theorem \ref{tm:uniform_estimate} or Theorem \ref{tm:graded_estimate}, we have the following estimate:
\begin{equation} \label{eq:energy-general}
\begin{aligned}
  &\quad  |(\bE_{\bsigma}, E_u)|_{\mathcal{A}} + a_{\rm s}(\bE_{\bp},\bE_{\bp})^{\frac12} \lesssim \|u - \Pi^{\rm SZ} u\|_{\widetilde{H}^s(\Omega)}  \\
         &+
         \left\{
         \begin{aligned}
         \Big(\sum_{T \in\mathcal{T}_h}& C_{11}^{-1}\|\bp - \bPi_{\bQ}^0 \bp\|^2_{L^2(\partial T)} \Big)^{\frac12}  \qquad\qquad\qquad\qquad\qquad\qquad~ \text{if } \bSigma_h = \bQ_h, \\
         \Big(\sum_{T \in\mathcal{T}_h}& C_{11}^{-1}\|\bp - \bPi_{\bQ}^0 \bp\|^2_{L^2(\partial T)} + \sum_{T\in \mathcal{T}_h} C_{\rm s}\|\bp - \bPi_\bSigma^0 \bp\|_{L^2(T)}^2 \Big)^{1/2} ~ \text{if }  \bSigma_h \subsetneq \bQ_h.
         \end{aligned}
         \right. 
\end{aligned}
\end{equation}
\end{lemma}
\begin{proof}
Taking \(\btau_h = \bE_\bsigma\), \(\bq_h = \bE_\bp\), and \(v_h = E_u\) in \eqref{eq:energy_primal}, and summing the three equations, we derive the following energy identity:
\begin{equation}\label{eq:energy_equality_ldg}
	\begin{aligned}
		|(\bE_{\bsigma},E_u)|^2_{\mathcal{A}} + a_{\rm s}(\bE_{\bp}, \bE_{\bp})  &= a(\bE_\bsigma ,\bE_\bsigma) + a_{\rm s}(\bE_\bp, \bE_\bp) + c(E_u,E_u)\\
		& = \underbrace{a(\be_\bsigma,\bE_\bsigma)}_{I_1} - \underbrace{(\be_\bp,\bE_\bsigma)}_{I_2}
		+\underbrace{(\be_\bsigma, \bE_\bp)
		+ b(e_u ,\bE_\bp)}_{I_3}  \\
		& \quad + \underbrace{a_{\rm s}(\bPi\bp, \bE_\bp)}_{I_4} - \underbrace{b(E_u,\be_\bp)}_{I_5}+ \underbrace{c(e_u, E_u)}_{I_6}.
	\end{aligned}
\end{equation}
Next, we estimate the terms $I_i$ for $i = 1,\ldots, 6$.
\begin{itemize}
    \item Estimate of $I_1$: Noting that $\be_{\bsigma} := \nabla (\Pi^{\rm SZ} u  - u) = \nabla e_u$, then by the H\"older inequality and and Lemma \ref{lem:riesz_grad} (equivalent form of $\widetilde{H}^s$-norm), we obtain
    \[
    \begin{aligned}
        |I_1| &= \vert a(\be_{\bsigma},\bE_\bsigma)\vert \leq a(\bE_\bsigma,\bE_\bsigma)^{\frac12} \|I_{1-s}\nabla e_u\|_{L^2(\mathbb{R}^n)}\\
        &\leq Ca(\bE_\bsigma,\bE_\bsigma)^{\frac12} \|e_u\|_{\widetilde{H}^s(\Omega)} 
        \leq \frac{1}{2}a(\bE_\bsigma,\bE_\bsigma) + \tilde{C} \|e_u\|^2_{\widetilde{H}^s(\Omega)}.
    \end{aligned}
    \]
    \item Estimate of $I_2$: Since $\be_{\bp}= \bPi^0 \bp - \bp$ and $\bSigma(T) \subset \bQ(T)$, we have $I_2 = 0$. 
    \item Estimate of $I_3$: Since $e_u = u - \Pi^{\rm SZ} u$ is continuous and vanishes on $\partial \Omega$, and $\be_{\bsigma} = \nabla e_u$, we have:
    \[
    \begin{aligned}
        I_3 &= (\be_\bsigma,\bE_\bp)+b(e_u,\bE_\bp)\\
        &= \int_{\mathcal{F}_h^0}(\vavg{\bE_\bp} - \bbeta[\bE_\bp])\cdot \llbracket e_u\rrbracket \mathrm{d}s +\int_{\partial \Omega}e_u\bE_\bp\cdot \bn\mathrm{d}s = 0.
    \end{aligned}
    \]
    \item Estimate of $I_4$: When $\bSigma_h(T)=\bQ(T)$, the term $I_4$ vanishes identically. In contrast, if $\bSigma_h(T)\subsetneq \bQ(T)$, we obtain
\[
\| \bp - \bPi_\bSigma^0 \bp \|_{L^2(T)}^2 
  = \| \bp - \bPi_\bQ^0 \bp \|_{L^2(T)}^2 
  + \| (\bI - \bPi_\bSigma^0) \bPi_\bQ^0 \bp \|_{L^2(T)}^2.
\]
Therefore,
\[
\begin{aligned}
  |I_4| 
    &\le \frac{1}{2} a_{\rm s}(\bE_p, \bE_p) 
         + C a_{\rm s}(\bPi_\bQ^0 \bp, \bPi_\bQ^0 \bp) \\
    &\le \frac{1}{2} a_{\rm s}(\bE_p, \bE_p) 
         + C \sum_{T \in \mathcal{T}_h} C_s \, \|\bp - \bPi_\bSigma^0 \bp\|_{L^2(T)}^2.
\end{aligned}
\]
    \item Estimate of $I_5$: Since $\be_{\bp}= \bPi \bp - \bp$ is locally orthogonal to $\nabla V_h$, we then have
        \[
        \begin{aligned}
            |I_5| &= |b(E_u,\be_\bp)| = |(\nabla E_u,\be_\bp) -\int_{\mathcal{F}_h} (\vavg{\be_\bp}-\bbeta[\be_\bp])\cdot\llbracket E_u\rrbracket\,\mathrm{d}s|\\
            &= \left|\int_{\mathcal{F}_h} (\vavg{\be_\bp}-\bbeta[\be_\bp])\cdot\llbracket E_u\rrbracket\,\mathrm{d}s\right|.
        \end{aligned}
        \]
        It follows directly from the definition of $c(\cdot,\cdot)$ that
        $$
        |I_5| \leq \frac{1}{2}c(E_u,E_u) + C\sum_{T\in\mathcal{T}_h}C_{11}^{-1}\|\be_\bp\|^2_{L^2(\partial T)}.
        $$ 
        \item Estimate of $I_6$: We easily see that $\llbracket e_u\rrbracket|_{\mathcal{F}_h} = \llbracket u - \Pi^{\rm SZ} u \rrbracket|_{\mathcal{F}_h} = 0$ hence $I_6 = 0$.
\end{itemize}

Combining these estimates and moving the remaining terms to the left-hand side yields the desired estimates. 
\end{proof}

\begin{proof}[Proof of energy estimates in Theorem \ref{tm:uniform_estimate} and Theorem \ref{tm:graded_estimate}] 
We always take \\
$V(T) = \mathcal{P}_1(T)$ and $\boldsymbol{\Sigma}(T) = \boldsymbol{\mathcal{P}}_0(T)$.

\underline{On quasi-uniform meshes}, we take $C_{11} \eqsim h^{1-2s}$ and $\bQ(T)=\boldsymbol{\mathcal{P}}_0(T)$. 
Combining the Sobolev regularity estimate \eqref{eq:sobolev-regularity-all} with the approximation 
properties of the interpolation operators \eqref{eq:SZ-uniform} and \eqref{eq:Pi-partial-T} yields
\begin{equation} \label{eq:energy-uniform}
\begin{aligned}
\|u - \Pi^{\rm SZ}u \|_{\widetilde{H}^s(\Omega)} &\lesssim h^{\frac12-\varepsilon} \|u\|_{H^{s+\frac12 -\varepsilon}(\Omega)} \lesssim \frac{h^{\frac12-\varepsilon}} {\sqrt{\varepsilon}} \|f\|_{L^2(\Omega)}. \\
\Big(\sum_{T\in\mathcal{T}_h} C_{11}^{-1}\|\bp - \bPi_{\bQ}^0 \bp\|^2_{L^2(\partial T)} \Big)^{\frac12} & \lesssim h^{\frac12-\varepsilon} \|\bp \|_{H^{\frac32-s -\varepsilon}(\Omega)}
\lesssim \frac{h^{\frac12-\varepsilon}} {\sqrt{\varepsilon}} \|f\|_{L^2(\Omega)}.
\end{aligned}
\end{equation}
We further note that, in this case, $\bQ(T)=\boldsymbol{\mathcal{P}}_0(T)$ already makes full use of the Sobolev regularity of $\bp$ since $\frac32 - s < 1$.

\underline{On graded meshes}, we take $\bQ(T)=\boldsymbol{\mathcal{P}}_1(T)$, and the condition on $\beta$ together with the choice of $\mu$ is given by \eqref{eq:parameter_condition}.
We now select the regularity parameters optimally, leading to the following error estimate.
\begin{itemize}
\item In view of Theorem \ref{thm:weighted_sobolev_estimate} (weighted Sobolev regularity of $u$),  
the parameters are chosen as $t = 2 - \varepsilon$ and $\gamma = \frac32 - s - \frac{3-2s}{4-2s}\varepsilon$ for $n = 1$,  
while for $n \ge 2$ the choice $t = s + \frac{n}{2(n-1)} - n\varepsilon$ and 
$\gamma = \frac{1}{2(n-1)} - \varepsilon$ is adopted.  
With these choices, \eqref{eq:SZ-graded} implies that
\begin{equation} \label{eq:energy-u-graded}
\|u - \Pi^{\rm SZ}u\|_{\widetilde{H}^s(\Omega)} \lesssim 
\begin{cases}
\frac{1}{\varepsilon}h^{2-s-\varepsilon} \|f\|_{C^\beta(\overline{\Omega})} & \text{if }n=1,\\
\frac{1}{\varepsilon} h^{\frac{n}{2(n-1)} - n\varepsilon} \|f\|_{C^\beta(\overline{\Omega})} & \text{if }n\geq 2. 
\end{cases}
\end{equation}

\item Recall that $\bPi_{\bQ}^{0}$ is the local $L^{2}$-projection onto $\boldsymbol{\mathcal{P}}_{1}(T)$, and that
$C_{11}^{-1/2}|_F \eqsim h_{T}^{s-\frac12}$ for $F \subset \partial T$.  
We distinguish two types of elements in \eqref{eq:energy-general}. If $T\cap\partial\Omega\neq\varnothing$ (elements touching the boundary), then \eqref{eq:Pi-1-weighted} yields
$$
C_{11}^{-1/2}\|\bp - \bPi_{\bQ}^0 \bp\|_{L^2(\partial T)}  \lesssim
    h_T^{s -1 + t - \gamma} |\bp|_{H_{\gamma}^{t}(T)} \eqsim h^{\mu(s+t-1-\gamma)}|\bp|_{H_{\gamma}^{t}(T)}.
$$ 
If $T\cap\partial\Omega=\varnothing$ (interior elements), then \eqref{eq:Pi-1-weighted} yields
$$
C_{11}^{-1/2}\|\bp - \bPi_{\bQ}^0 \bp\|_{L^2(\partial T)}  \lesssim  h^{s+t-1} \delta(T, \partial \Omega)^{ \frac{\mu - 1}{\mu}(s+t-1)}|\bp|_{H^t(T)}.
$$ 
From the above estimates, the balancing condition yields $\gamma = \frac{\mu - 1}{\mu}(s+t-1)$.
In view of Theorem \ref{thm:weighted_sobolev-p} (weighted Sobolev regularity of $\bp$), the constraint $t < \gamma + \tfrac{3}{2} - s$
leads to the following optimal choices of parameters:
 $t = 3-2s-\varepsilon$, $\gamma = \frac32 -s -\frac{3-2s}{4-2s}\varepsilon$ for $n=1$, whereas for $n\ge 2$ we choose $t = 1 - s + \frac{n}{2(n-1)} - n\varepsilon$ and 
 $\gamma = \frac{1}{2(n-1)} -\varepsilon$. With these choices, we have
\begin{equation} \label{eq:energy-p-graded}
\begin{aligned}
 \Big( \sum_{T \in \mathcal{T}_h}  C_{11}^{-1}\|\bp - \bPi_{\bQ}^0 \bp\|_{L^2(\partial T)}^2 \Big)^{\frac12} &
 \lesssim 
\begin{cases}
h^{2-s-\varepsilon} |\bp|_{H^t_\gamma} & \text{if }n=1,\\
h^{\frac{n}{2(n-1)} - n \varepsilon} |\bp|_{H^t_\gamma} & \text{if }n\geq 2,
\end{cases} \\
(\text{by } \eqref{eq:weighted-regularity-p}) \quad
& \lesssim \begin{cases}
\frac{1}{\varepsilon}h^{2-s-\varepsilon} \|f\|_{C^\beta(\overline{\Omega})} & \text{if }n=1,\\
\frac{1}{\varepsilon} h^{\frac{n}{2(n-1)} - n \varepsilon} \|f\|_{C^\beta(\overline{\Omega})} & \text{if }n\geq 2. 
\end{cases}
\end{aligned}
\end{equation}
\item  For the last term, note that
  $\bSigma(T)=\boldsymbol{\mathcal{P}}_{0}(T)$. Moreover, recall that
    $C_s|_T \eqsim h_T^{2-2s}$ for $n=1$, while
 $C_s|_T \eqsim h_T^{2\theta}$ with $\theta \geq \max\{\frac{n}{2(n-1)}, s\} -1$ for $n\ge2$. We invoke \eqref{eq:Pi-0-weighted} directly and choose the parameters as follows:
 $\ell = 1-\varepsilon$,  $\gamma = \frac{3}{2} - s - \frac{3-2s}{4-2s}\,\varepsilon$ for $n=1$,
 and $\ell = 1 - \max\{0, s-\frac{n}{2(n-1)}\} - n\varepsilon$, 
 $\gamma = \frac{1}{2(n-1)} -\varepsilon$ for $n\ge2$. With these choices, we obtain
 \begin{equation} \label{eq:energy-sigma-graded}
\begin{aligned}
 \Big(\sum_{T \in \mathcal{T}_h} C_{s} \|\bp - \bPi_{\bSigma}^0 \bp\|_{L^2(T)}^2 \Big)^{\frac12}
 & \lesssim 
\begin{cases}
h^{2-s-\varepsilon} |\bp|_{H^\ell_\gamma} & \text{if }n=1,\\
h^{\frac{n}{2(n-1)} - n\varepsilon} |\bp|_{H^\ell_\gamma} & \text{if }n\geq 2,
\end{cases} \\
(\text{by } \eqref{eq:weighted-regularity-p}) \quad \lesssim &
\begin{cases}
\frac{1}{\varepsilon}h^{2-s-\varepsilon} \|f\|_{C^\beta(\overline{\Omega})} & \text{if }n=1,\\
\frac{1}{\varepsilon}h^{\frac{n}{2(n-1)} - n\varepsilon} \|f\|_{C^\beta(\overline{\Omega})} & \text{if }n\geq 2.
\end{cases}
\end{aligned}
\end{equation}
\end{itemize}

Upon taking $\varepsilon = |\log h|^{-1}$, Lemma \ref{lm:energy_estimate_triangle_LDG} and \eqref{eq:energy-uniform} immediately yields \eqref{eq:energy-result-uniform}. On the other hand, \eqref{eq:energy-u-graded}, \eqref{eq:energy-p-graded} and \eqref{eq:energy-sigma-graded} lead to \eqref{eq:energy-result-graded}.
\end{proof}

\subsection{$L^2$ error estimate} To derive the error estimate for \(\|u - u_h\|_{L^2(\Omega)}\), we employ a duality argument by introducing the following adjoint problem: Find \(\varphi\) such that
\begin{equation}\label{eq:adjoint_problem}
    \begin{aligned}
        (-\Delta)^s\varphi &= u - u_h \quad \text{in } \Omega,\\
        \varphi &= 0 \quad \text{on } \Omega^c.
    \end{aligned}
\end{equation}
This adjoint problem will serve as the foundation for the \(L^2\)-error analysis through duality techniques, allowing us to relate the approximation error to known regularity results for fractional elliptic problems. Let \(\bPhi = \nabla\varphi\) and \(\bPsi = I_{2-2s}\bPhi\).

For the dual problem, we again employ the Theorem \ref{thm:sobolev_regularity} (Sobolev regularity on Lipschitz domains) with the right-hand side $u-u_h \in L^2(\Omega)$, yielding
\begin{equation}\label{eq:epsilon_right_hand_side_adjoint}  
\|\varphi\|_{H^{s+\frac12-\varepsilon}(\Omega)} + \|\bPhi\|_{H^{s-\frac{1}{2}-\varepsilon}(\Omega)} + \|\bPsi\|_{H^{\frac{3}{2}-s-\varepsilon}(\Omega)} \leq \frac{1}{\sqrt{\varepsilon}} \|u-u_h\|_{L^2(\Omega)}.  
\end{equation}

 Following consistency results analogous to \eqref{eq:mixed_fractional_problem} and \eqref{eq:fracDG_consistency}, we obtain the auxiliary problem:
\begin{equation}\label{eq:auxiliary_problem}
\left\{
\begin{array}{r@{\;}c@{\;}l @{\;}c@{\;}l @{\;}c@{\;}l @{\quad}l}
  a(\bPhi,\btau) & - & (\bPsi,\btau) &   &                  & = & 0          & \forall \btau \in \bSigma_h+\bSigma,\\
  (\bPhi,\bq)    &   &               & + & b(\varphi,\bq)   & = & 0          & \forall \bq \in \bQ_h+\bQ,\\
                 & - & b(v,\bPsi)    & + & c(\varphi,v)     & = & (u-u_h,v)  & \forall v \in V_h+V .
\end{array}
\right.
\end{equation}
Next, we take \(\Pi \varphi = \Pi^{\rm SZ}\varphi\), \(\bPi \bPhi = \nabla \Pi^{\rm SZ} u\), and \(\bPi \bPsi = \bPi_\bQ^0\bPsi\). Additionally, to simplify notation, we introduce the following abbreviations:  
\[
\be_{\bPhi} := \nabla (\Pi^{\rm SZ}\varphi - \varphi),  ,\quad\be_{\bPsi} := \bPi_{\bQ}^0\bPsi - \bPsi, \quad e_{\varphi} := \Pi^{\rm SZ}\varphi - \varphi.
\]

For the auxiliary problem \eqref{eq:auxiliary_problem}, taking test functions $\btau = \bsigma-\bsigma_h\), \(\bq = -(\bp-\bp_h)$, and $v = u-u_h$, then summing the three equations and applying the error equation \eqref{eq:error_equation} yields the $L^2$-error identity:
\begin{equation}\label{eq:L2_estimate_identity}
    \begin{aligned}
        \|u-u_h\|_{L^2(\Omega)}^2 
        &= -b(\varphi, \bp-\bp_h) + c(\varphi,u-u_h) + a(\bPhi, \bsigma - \bsigma_h) - \left(\bPhi,\bp - \bp_h\right) \\
        &\quad - \left[(\bsigma-\bsigma_h,\bPsi) + b(u-u_h,\bPsi)\right] \\
        &= \underbrace{b(e_{\varphi}, \bp-\bp_h)}_{J_1} - \underbrace{c(e_{\varphi},u-u_h)}_{J_2} 
        - \underbrace{a(\bsigma - \bsigma_h, \be_{\bPhi})}_{J_5} + \underbrace{\left(\bp - \bp_h, \be_{\bPhi}\right)}_{J_6} \\
        &\quad +\underbrace{(\bsigma-\bsigma_h,\be_{\bPsi})}_{J_3} + \underbrace{b(u-u_h,\be_{\bPsi})}_{J_4} - \underbrace{a_{\rm s}(\bp_h, \bPi_\bQ^0 \bPsi)}_{J_7}.
    \end{aligned}
\end{equation}

    Starting from the identity \eqref{eq:L2_estimate_identity}, we only need to estimate the terms $\{J_i\}_{i=1}^{7}$:
    \begin{itemize}
        \item Estimates of $J_1+J_6$ and $J_2$: Since $ \be_{\bPhi} = \nabla e_\varphi$, and $e_\varphi = \Pi^{\rm SZ} \varphi - \varphi$ is continuous and vanishes on \(\partial \Omega\), we have
        \[
        \begin{aligned}
        J_1+J_6 &= - \int_{\mathcal{F}_h}\llbracket \varphi - \Pi^{SZ} \varphi \rrbracket \cdot \left(\vavg{\bp-\bp_h} - \bbeta[\bp-\bp_h] \right)\mathrm{d}s = 0,\\
        J_2 &=  c(u-u_h, e_\varphi) = 0.
        \end{aligned}        
        \]
        \item Estimate of $J_5$: Note that $a(\bsigma-\bsigma_h,\be_{\bPhi}) = (I_{2-2s}(\bsigma - \bsigma_h), \nabla e_\varphi)$. By H\"older's inequality and Lemma \ref{lem:riesz_grad} (equivalent form of $\widetilde{H}^s$-norm), approximation property \eqref{eq:SZ-uniform} and Sobolev regularity 
        \eqref{eq:epsilon_right_hand_side_adjoint}, we obtain
        \[
        \begin{aligned}
            |J_5|
            & \leq a(\bsigma - \bsigma_h,\bsigma - \bsigma_h)^{\frac12} \|I_{1-s}\nabla e_\varphi \|_{\boldsymbol{L}^2(\mathbb{R}^n)} \\
            & \leq  Ca(\bsigma - \bsigma_h, \bsigma - \bsigma_h)^{\frac12}\|e_{\varphi}\|_{\widetilde{H}^s(\Omega)}\\
            & \leq  \frac{C}{\sqrt{\varepsilon}}h^{\frac12-\varepsilon}a(\bsigma - \bsigma_h, \bsigma - \bsigma_h)^{1/2}\|u-u_h\|_{L^2(\Omega)}.
        \end{aligned}
        \]
        \item Estimate of $J_3+J_4$: Note that $\bsigma = \nabla u$, and $(\be_{\bPsi}, \bsigma_h) = (\be_{\bPsi}, \nabla_h u_h) = 0$ since $\bPi_Q^0$ is the local $L^2$ projection, we obtain:
        \[
        J_3+J_4 =  \int_{\mathcal{F}_h} \left(\vavg{\be_{\bPsi}} - \bbeta[\be_{\bPsi}]\right) \cdot \llbracket u-u_h \rrbracket \mathrm{d}s.
        \]
        Thus, by the approximation property \eqref{eq:Pi-uniform} and Sobolev regularity \eqref{eq:epsilon_right_hand_side_adjoint},
        \[
        \begin{aligned}
        |J_3+J_4| &\lesssim c(u-u_h,u-u_h)^{\frac12}\Big(\sum_{T\in \mathcal{T}_h} C_{11}^{-1} \|\be_{\bPsi}\|^2_{L^2(\partial T)}\Big)^{\frac12}\\
        &\leq \frac{C}{\sqrt{\varepsilon}} h^{\frac12-\varepsilon} c(u-u_h,u-u_h)^{1/2}\|u-u_h\|_{L^2(\Omega)}.
        \end{aligned} 
        \]
        \item Estimate of $J_7$: When $\bSigma_h=\bQ_h$, the term $J_7$ vanishes identically. In the case $\bSigma_h \subsetneq \bQ_h$, we have
        $$
        \begin{aligned}
        a_{\rm s}(\bPi_Q^0 \bPsi, \bPi_Q^0 \bPsi)^{\frac12} \leq& \Big( \sum_{T\in \mathcal{T}_h}C_{\rm s} \|\bPsi - \bPi_{0}^0 \bPsi\|_{L^2(T)}^2 \Big)^{\frac12}\\
        \lesssim & 
        \begin{cases}
        h^{1-s} h^{\frac32 - s -\varepsilon} \|\bPsi\|_{H^{\frac32 - s -\varepsilon}(\Omega)} & \text{if } n=1 \\
        h^{-1 + \max\{\frac{n}{2(n-1)}, s\}} h^{\frac32 - s -\varepsilon} \|\bPsi\|_{H^{\frac32 - s -\varepsilon}(\Omega)} & \text{if } n\geq 2
        \end{cases} \\
        \lesssim& \frac{1}{\sqrt{\varepsilon}} h^{\frac12-\varepsilon} \|u-u_h\|_{L^2(\Omega)}.
        \end{aligned}
        $$ 
        Therefore, when $\bSigma_h \subsetneq \bQ_h$, it follows that
        $$
        |J_7| \lesssim a_{\rm s}(\bp_h, \bp_h)^{\frac12} \frac{1}{\sqrt{\varepsilon}} h^{\frac12-\varepsilon} \|u-u_h\|_{L^2(\Omega)}.
        $$    
    \end{itemize} 
    
    Combining all the above bounds, it follows that, on both quasi-uniform and graded meshes,  
the $L^{2}$-error exhibits an improvement of order $h^{\frac12}|\log h|^{\frac12}$ over the energy error.  
Consequently, the $L^{2}$-error estimates stated in  
Theorem~\ref{tm:uniform_estimate} and Theorem~\ref{tm:graded_estimate}, namely  
\eqref{eq:L2-result-uniform} and \eqref{eq:L2-result-graded}, are obtained.

\begin{remark}[sharpness of $L^2$ error estimate]
As in the discussion in \cite[Section 3]{borthagaray2021local}, the duality argument still relies on the unweighted estimate \eqref{eq:epsilon_right_hand_side_adjoint}.  
This is due to the fact that only $L^2$ data estimates are currently available for the adjoint problem.  
The numerical results in Section \ref{sc:fracDG-numerical} indicate that the $L^2$ error estimate on quasi-uniform meshes is sharp, whereas for graded meshes the observed $L^2$-convergence rate appears to be better than the theoretical prediction.
\end{remark}

\section{Numerical experiments}\label{sc:fracDG-numerical}
In this section, we present numerical experiments demonstrating the performance of the LDG method.
For the convergence rate tests, problems posed on the one- and two-dimensional unit ball
$\Omega = B_1(0)$ are considered, with the right-hand side chosen as $f=1$. In this setting, the exact solution admits the expression:
\begin{equation} \label{eq:unit-circle}
u = K_{n,s} (1-|x|^2)^s_+, \quad
\bsigma = -2 sK_{n,s}  x (1-|x|^2)^{s-1}_+, \quad 
\bp = I_{2-2s}\bsigma= -\tfrac{1}{n} x,
\end{equation}
where $K_{n,s}=\frac{2^{-2s}\Gamma(n/2)}{\Gamma(n/2+s)\Gamma(1+s)}$. Note that, due to the high symmetry of the computational domain, the $\bp$ does not exhibit any boundary singularity in this setting. Singular behavior of $\bp$ will be observed in the final test. 

In the LDG numerical fluxes \eqref{eq:numer-flux}, we set $\bbeta = \boldsymbol{0}$ and $C_{11}|_F = h_F^{1-2s}$, where $F \in \mathcal{F}_h$ denotes an $(n-1)$-dimensional face. 
For $n=1$, $h_F$ is the averaged size of the adjacent cells, and for $n=2$, it is the edge length.
Since the energy error $|(\bsigma-\bsigma_h,\,u-u_h)|_{\mathcal{A}}$ is not directly accessible, an identity inspired by \cite{acosta2017fractional} is invoked to obtain a computable expression:
\[
    |(\bsigma-\bsigma_h,u-u_h)|^2_{\mathcal{A}} 
    = -2(\bp,\bsigma_h) + (f,u+u_h) - a_{\rm s}(\bp_h, \bp_h),
\]
where \eqref{eq:error_equation} and \eqref{eq:short_LDG_scheme} are employed. For the convergence rate tests, the energy and $L^2$ errors for $s = 0.6, 0.7, 0.8, 0.9$ are presented.

\subsection{One-dimensional convergence rate test}
The exact solution \eqref{eq:unit-circle} on the one-dimensional unit ball is employed in this test. For uniform meshes, the observed convergence rates of the energy and $L^{2}$ errors, displayed in Figure~\ref{fig:uniform1D_001}, are $\mathcal{O}(h^{1/2})$ and $\mathcal{O}(h)$, respectively, in agreement with the Theorem \ref{tm:uniform_estimate}.
\begin{figure}[!htbp]
    \centering
    \includegraphics[width=0.49\textwidth]{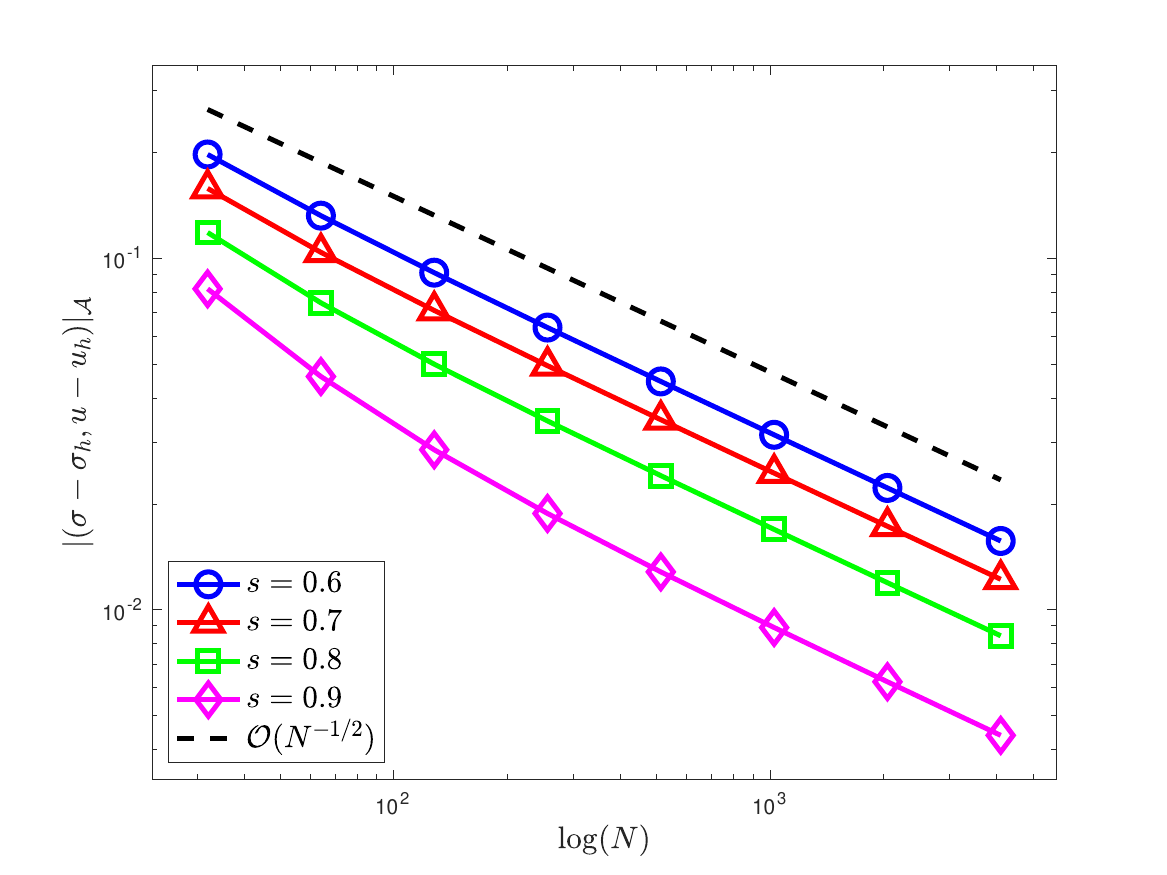}
    \includegraphics[width=0.49\textwidth]{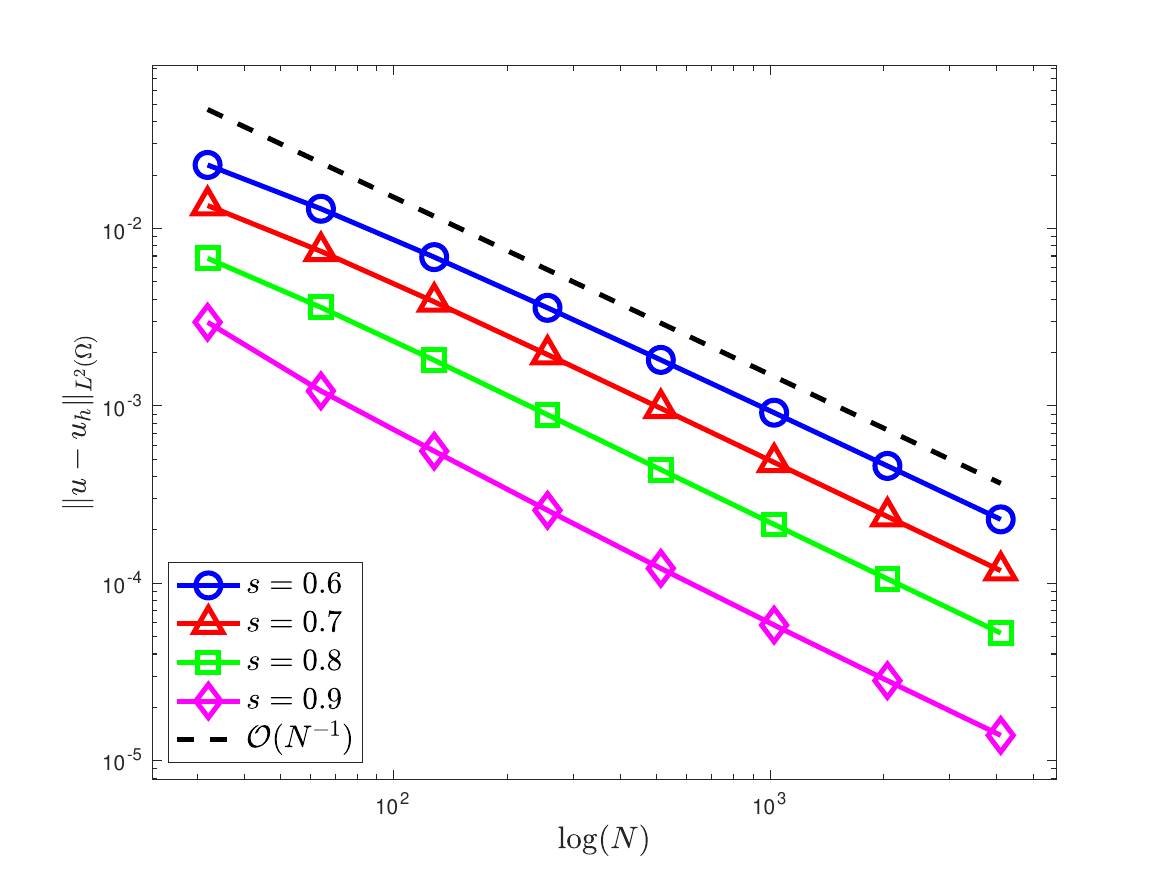}
    \vspace{-5mm}
    \caption{Convergence rates of $|(\bsigma-\bsigma_h,u-u_h)|_{\mathcal{A}}$ (left) and $\|u-u_h\|_{L^2(\Omega)}$ (right), 1D uniform meshes, $\bSigma(T) = \bQ(T)=\boldsymbol{\mathcal{P}}_0(T)$, $V(T) = \mathcal{P}_1(T)$.}
    \label{fig:uniform1D_001}
\end{figure}

For graded meshes, the condition of Theorem~\ref{tm:graded_estimate}
is adopted, with grading parameter $\mu = 4-2s$ (implying $h \eqsim N^{-1}$) and $C_{\rm s}|_T = h_T^{2-2s}$. The expected rates are $\mathcal{O}(N^{-2+s})$ for the energy error and $\mathcal{O}(N^{-5/2+s})$ for the $L^{2}$ error. As shown in Figure~\ref{fig:graded1D_001}, the energy error attains the predicted rate $\mathcal{O}(N^{-2+s})$, whereas the $L^2$ error achieves the unexpectedly sharp rate $\mathcal{O}(N^{-2})$, with indications of an even higher empirical rate.
\begin{figure}[!htbp]
    \centering
    \includegraphics[width=0.49\textwidth]{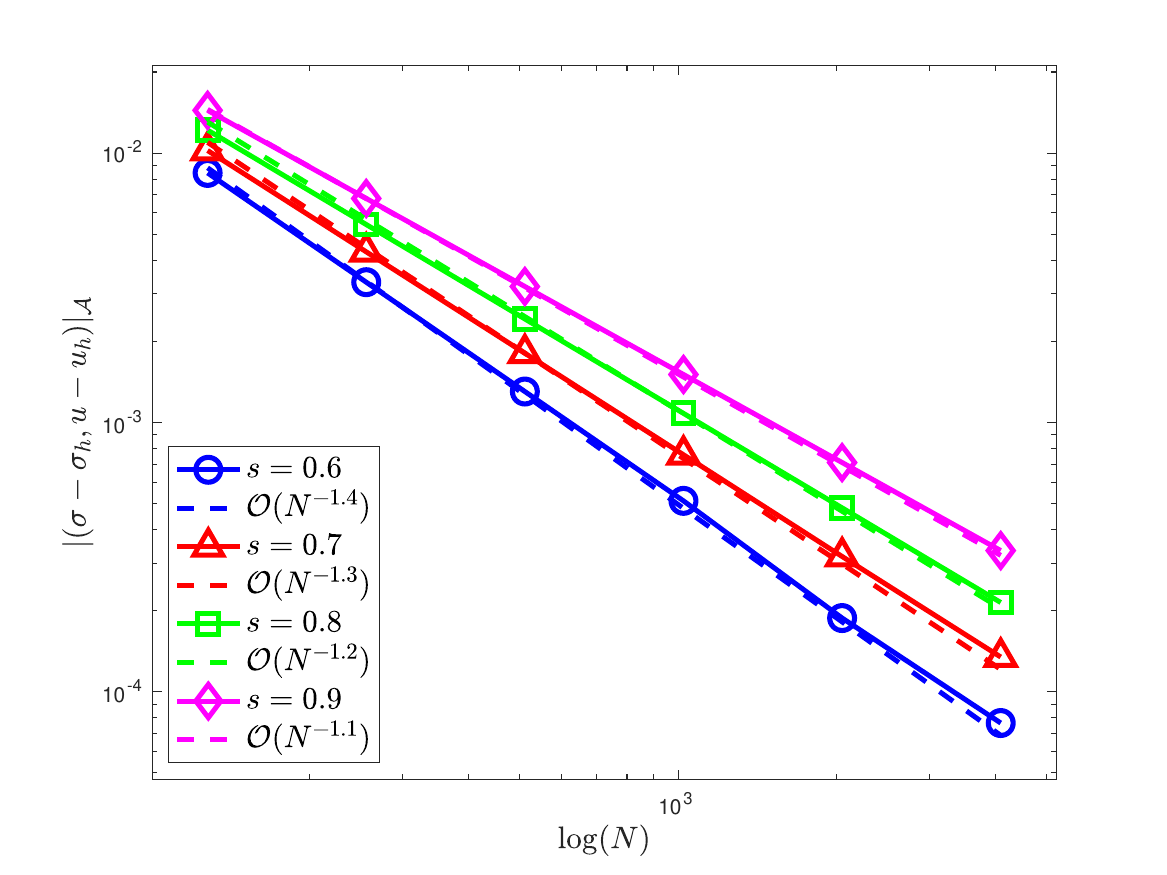}
    \includegraphics[width=0.49\textwidth]{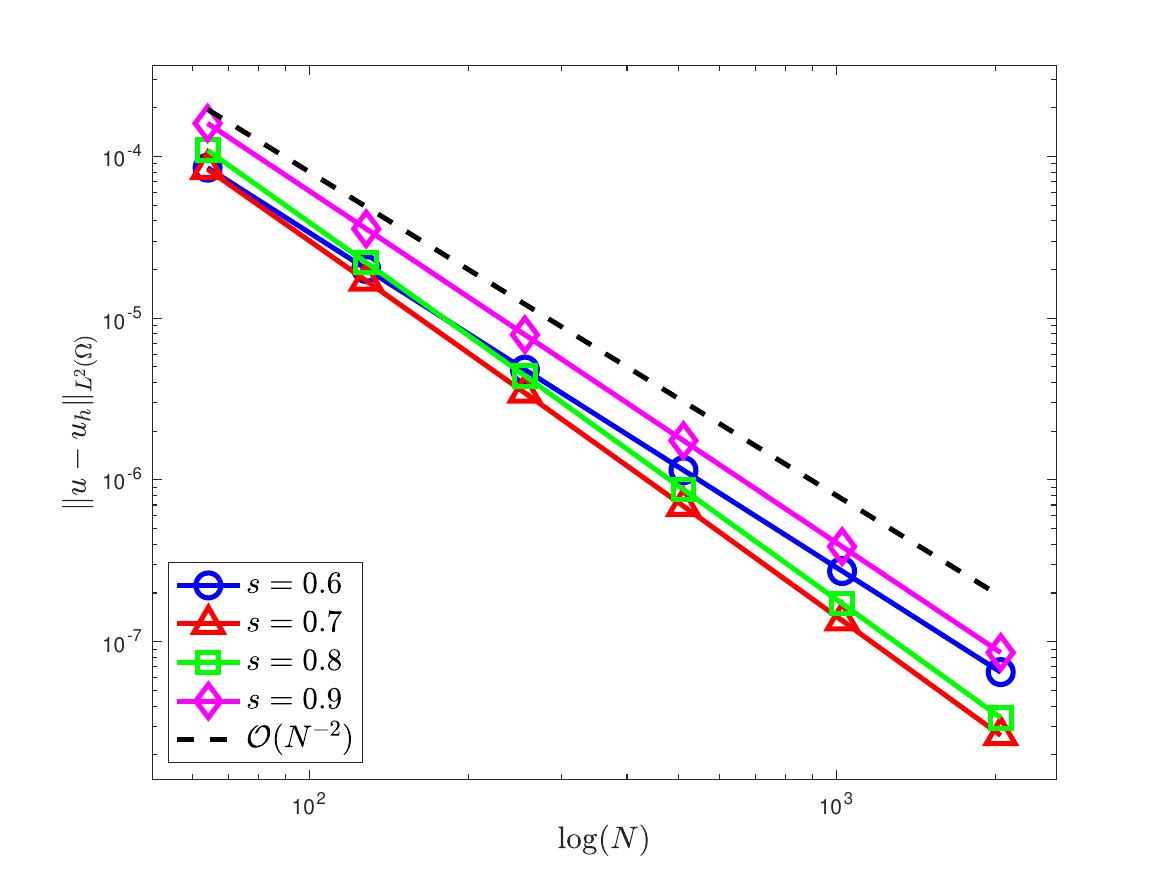}
    \vspace{-5mm}
    \caption{Convergence rates of $|(\bsigma-\bsigma_h,u-u_h)|_{\mathcal{A}}$ (left) and $\|u-u_h\|_{L^2(\Omega)}$ (right),
    1D graded meshes with $\mu = 4-2s$, $\bSigma(T) = \boldsymbol{\mathcal{P}}_0(T)$, $\bQ(T) = \boldsymbol{\mathcal{P}}_1(T)$, $V(T) = \mathcal{P}_1(T)$, $C_{\rm s}|_T = h_T^{2-2s}$.}
    \label{fig:graded1D_001}
\end{figure}

\subsection{Two-dimensional convergence rate test}
The exact solution \eqref{eq:unit-circle} on the two-dimensional unit ball is employed in this test. Figure \ref{fig:uniform2D_001} illustrates the results on uniform meshes. The observed convergence rates of the energy and $L^2$ errors are approximately $\mathcal{O}(N^{-1/4})$ and $\mathcal{O}(N^{-1/2})$, respectively, with slightly better performance in some cases (in particular for $s = 0.9$). Since $h = \mathcal{O}(N^{-1/2})$ on uniform meshes, these numerical results are essentially consistent with Theorem \ref{tm:uniform_estimate}.
\begin{figure}[htbp!]
    \centering
    \includegraphics[width=0.49\textwidth]{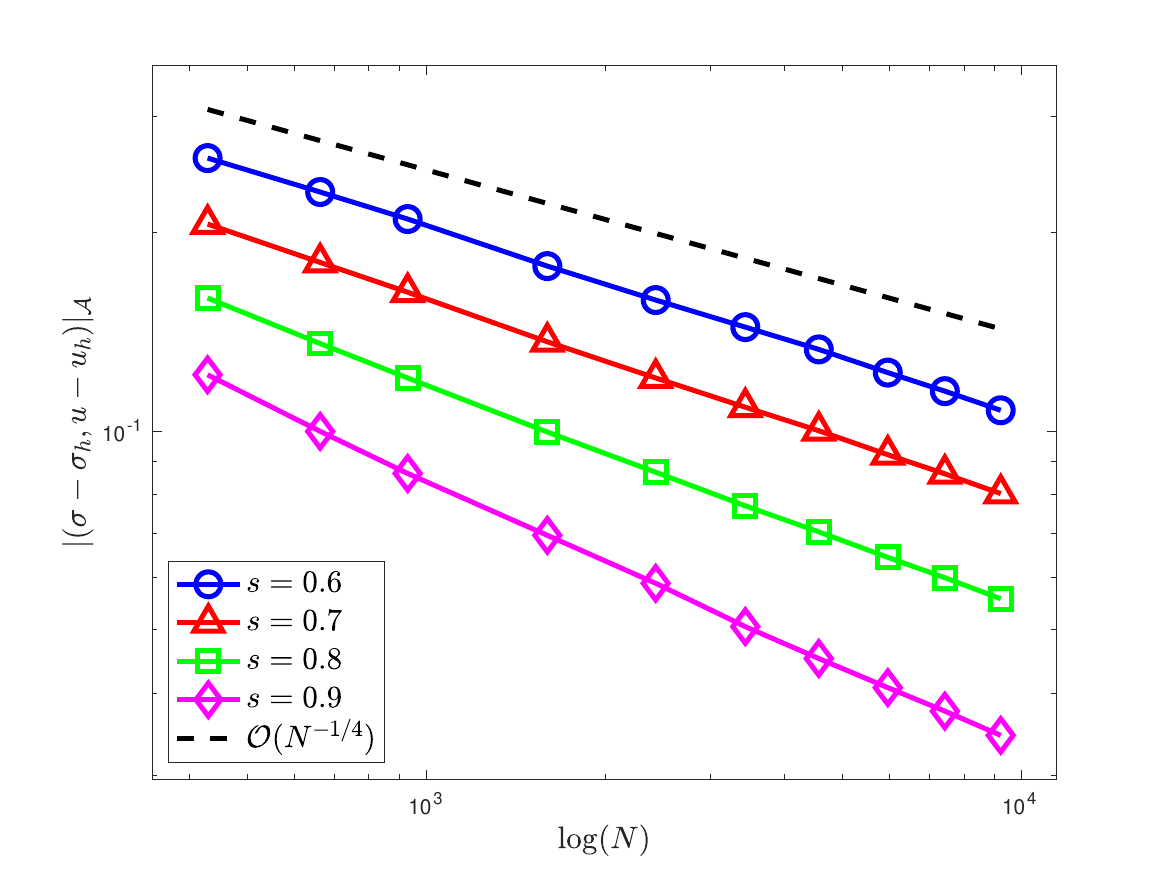}
    \includegraphics[width=0.49\textwidth]{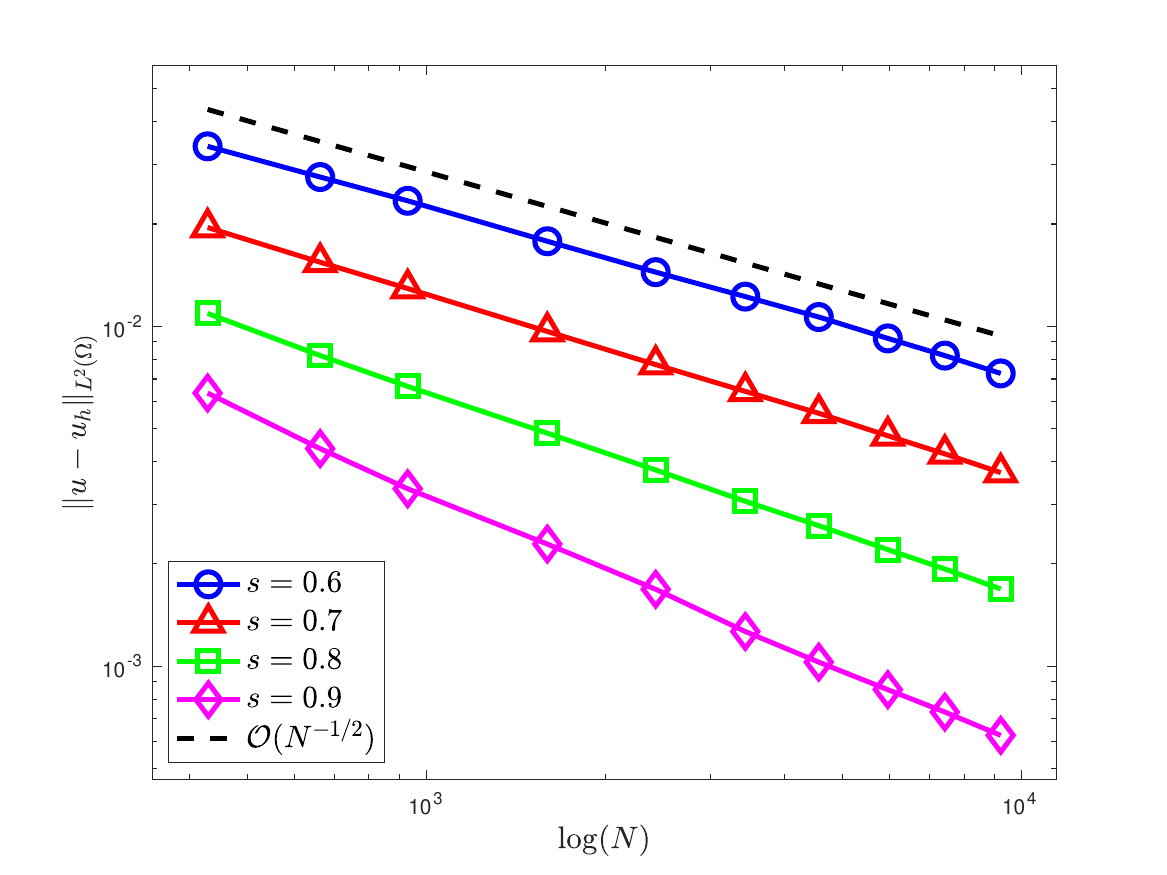}
        \vspace{-5mm}
    \caption{Convergence rates of $|(\bsigma-\bsigma_h,u-u_h)|_{\mathcal{A}}$ (left) and $\|u-u_h\|_{L^2(\Omega)}$ (right), 2D quasi-uniform meshes, $\bSigma(T) = \bQ(T)=\boldsymbol{\mathcal{P}}_0(T)$, $V(T) = \mathcal{P}_1(T)$.}
    \label{fig:uniform2D_001}
\end{figure}

For graded meshes, in accordance with conditions \eqref{eq:parameter_condition} and \eqref{eq:parameter-C}, we take the parameters $\mu = 2$ and $C_{\rm s} = 1$.
Theorem \ref{tm:graded_estimate} shows that the convergence rate of the energy error is $\mathcal{O}(N^{-1/2})$ (up to a logarithmic factor), while the rate of $\|u - u_h\|_{L^2(\Omega)}$ is $\mathcal{O}(N^{-3/4})$. From the numerical results in Figure \ref{fig:graded2D_001}, it can be observed that the energy error essentially attains the predicted convergence rate, whereas the performance of the $L^2$ error appears slightly better than the theoretical estimate, lying between $\mathcal{O}(N^{-3/4})$ and $\mathcal{O}(N^{-1})$.
\begin{figure}[!htbp]
    \centering
    \includegraphics[width=0.49\textwidth]{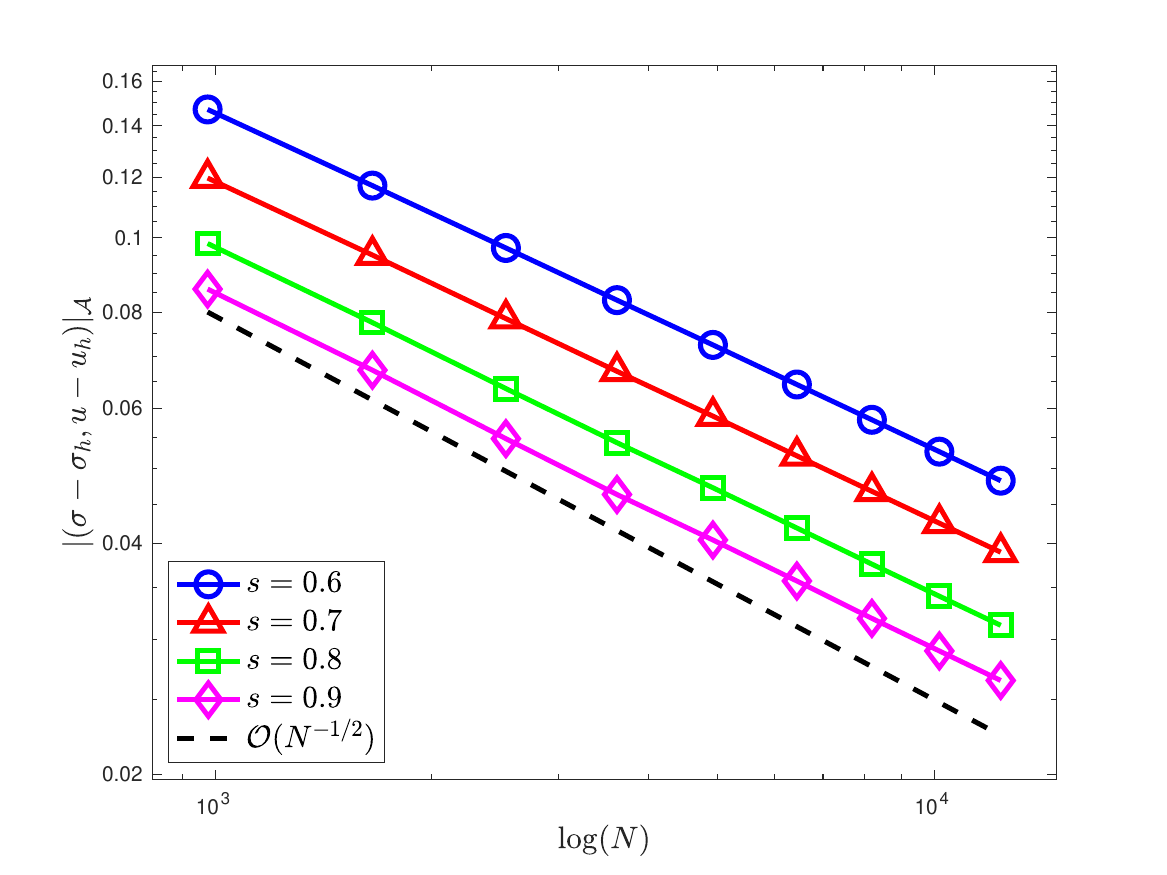}
    \includegraphics[width=0.49\textwidth]{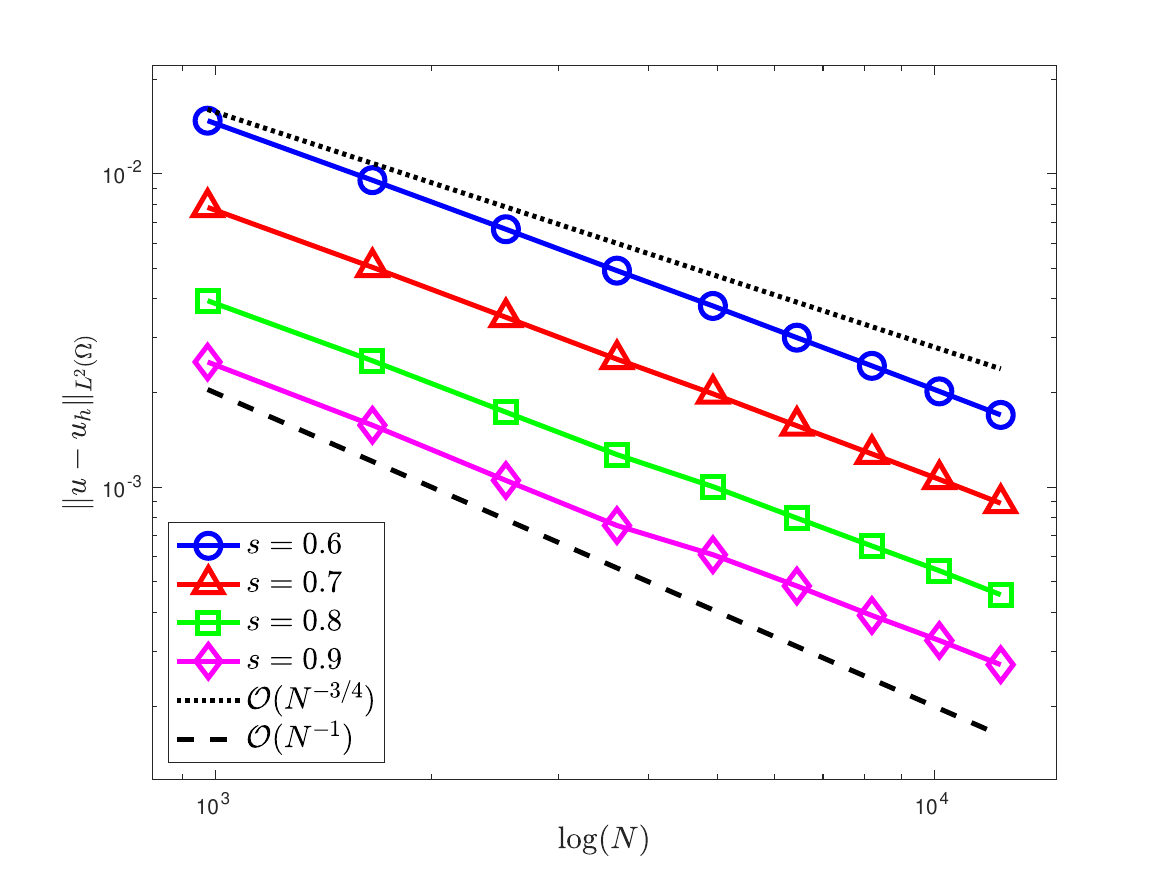}
    \vspace{-5mm}
    \caption{Convergence rates of $|(\bsigma-\bsigma_h,u-u_h)|_{\mathcal{A}}$ (left) and $\|u-u_h\|_{L^2(\Omega)}$ (right),
    2D graded meshes with $\mu = 2$, $\bSigma(T) = \boldsymbol{\mathcal{P}}_0(T)$, $\bQ(T) = \boldsymbol{\mathcal{P}}_1(T)$, $V(T) = \mathcal{P}_1(T)$, $C_{\rm s}|_T = 1$.}
    \label{fig:graded2D_001}
\end{figure}

\subsection{Performance of Riesz potential on square domain}
The final numerical experiment is performed on the two-dimensional computational domain $\Omega = (-1,1)^2$, where the geometry is no longer isotropic. The mesh is generated by applying newest vertex bisection to an initial coarse triangulation, and the marking criterion follows \cite{borthagaray2023robust}: a cell $T$ is marked whenever
\begin{equation*} \label{eq:marking}
|T| > \theta N^{-1} \log N \cdot d(x_T,\partial\Omega)^{2(\mu-1) / \mu},
\end{equation*}
where $x_T$ denotes the barycenter of $T$, with parameters $\theta = 24$ and $\mu = 2$. It is further recalled that $N := \dim V_h$ equals three times the number of elements.

The two-dimensional graded-mesh configuration is employed in this experiment, with
$\boldsymbol{\Sigma}(T)=\boldsymbol{\mathcal{P}}_{0}(T)$,
$\boldsymbol{Q}(T)=\boldsymbol{\mathcal{P}}_{1}(T)$,
$V(T)=\mathcal{P}_{1}(T)$,
and $C_{\mathrm{s}}|_{T}=1$.
After eleven bisection steps, resulting in $5,776$ elements, the first component of the computed discrete Riesz potential $\boldsymbol{p}_h$ is shown in Figure~\ref{fig:ph-square}. From Figure~\ref{fig:ph-sub1}, it is observed that, due to the non-spherical geometry, the first component of the Riesz potential behaves almost linearly in the $x$-direction, whereas its behavior in the $y$-direction resembles that of the typical one-dimensional fractional Laplacian profile. In particular, Figure~\ref{fig:ph-sub2} presents the slice at $x=-1$ for $s=0.6, 0.7, 0.8, 0.9$. The variation of the first component of $\boldsymbol{p}_h$ along the $y$-axis becomes increasingly pronounced near the boundary as $s$ increases, which is consistent with the estimate~\eqref{eq:riesz_grad_est} on the differentiability of the Riesz potential.

\begin{figure}[!htbp]
    \centering
    \begin{subfigure}[t]{0.49\textwidth} 
        \centering
        \includegraphics[width=\textwidth]{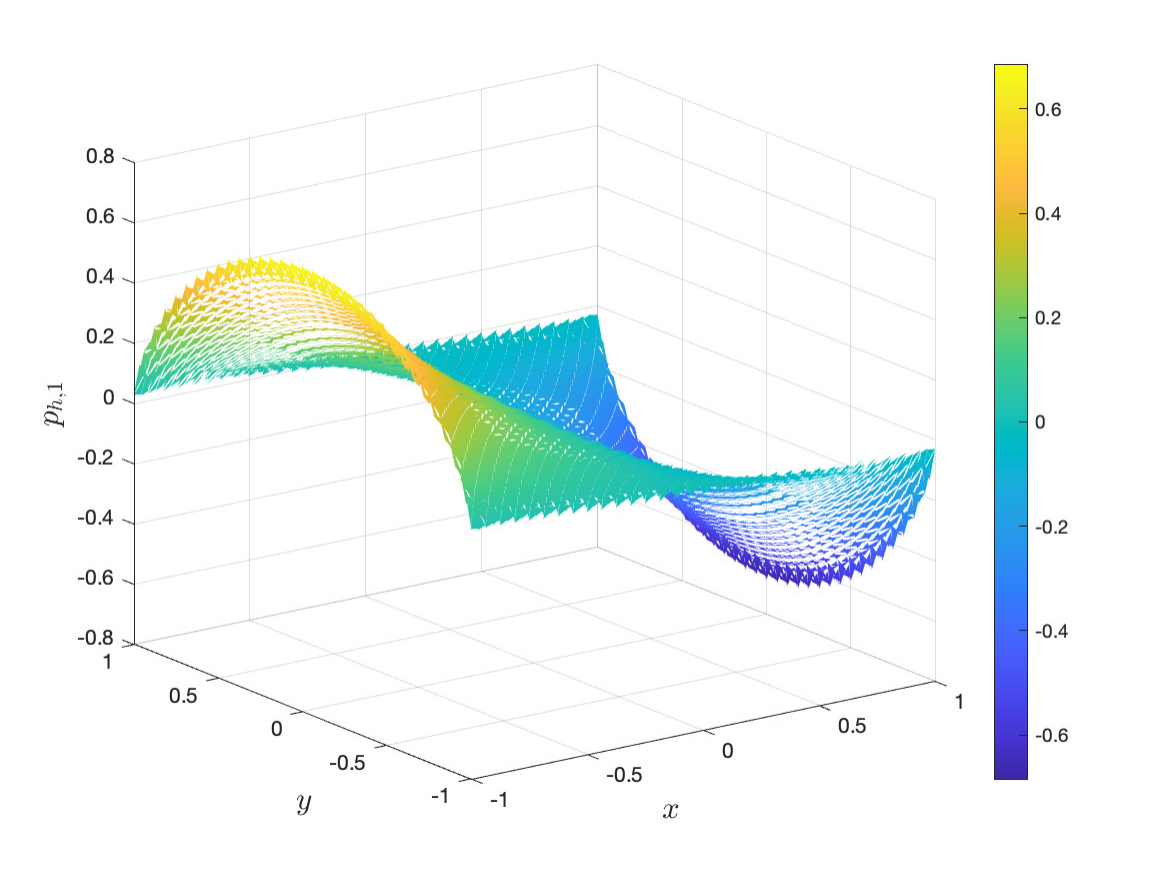} 
        \vspace{-2mm}
        \caption{$s = 0.9$, solution profile}
        \label{fig:ph-sub1}
    \end{subfigure}%
    \begin{subfigure}[t]{0.49\textwidth}
        \centering
        \includegraphics[width=\textwidth]{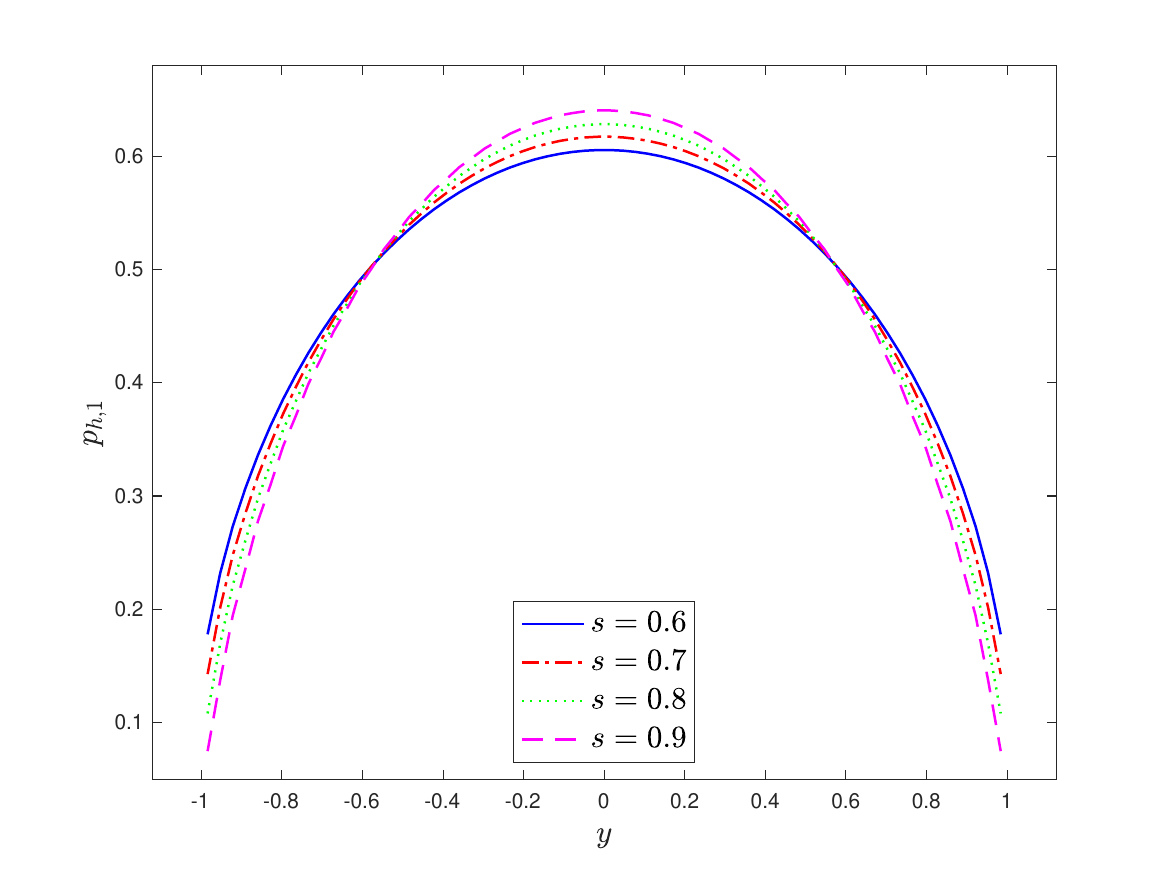}
              \vspace{-2mm}
        \caption{Slice at $x=-1$ for various values of $s$.}
        \label{fig:ph-sub2}
    \end{subfigure}
          \vspace{-3mm}
    \caption{First component of $\boldsymbol{p}_h$ on the domain $(-1,1)^2$ with the data $f=1$.} 
    \label{fig:ph-square}
\end{figure}

\bibliographystyle{siamplain}
\bibliography{frac.bib}
\end{document}